\documentclass[10pt]{amsart}
\usepackage{amsmath, amsthm, amsfonts}
\usepackage{a4wide}
\usepackage{amsmath, array}
\usepackage{url}
\usepackage{enumitem}
\usepackage{booktabs}
\usepackage{amssymb}
\usepackage{graphicx}
\usepackage{yhmath}
\usepackage{graphicx}
\usepackage{xcolor}
\usepackage{stackengine}
\usepackage{epsfig,graphics}
\usepackage{blkarray}
\usepackage[all,knot,poly,color]{xy}
\usepackage{etoolbox}
\patchcmd{\subsubsection}{\itshape}{\bfseries}{}{}

\usepackage{xcolor}
\usepackage[hyperfootnotes=false]{hyperref}
\hypersetup{
	colorlinks=false,
	citecolor=Violet,
	linkcolor=Red,
	urlcolor=Blue}

\newtheorem{thm}{Theorem}[section]
\newtheorem{cor}[thm]{Corollary}

\newtheorem{prop}[thm]{Proposition}
\theoremstyle{definition}
\newtheorem{defn}[thm]{Definition}
\newtheorem{rem}[thm]{Remark}

\newtheorem{ex}[thm]{Example}
\newtheorem{as}[thm]{Assumption}
\newtheorem{crr}[thm]{Criterion}

\def\ZZ{\mathbb{Z}}
\def\NN{\mathbb{N}}
\def\CC{\mathbb{C}}

\newcommand{\gra}{{\alpha}} \newcommand{\grb}{{\beta}} \newcommand{\grg}{{\gamma}} \newcommand{\grd}{{\delta}}
 \newcommand{\grz}{{\zeta}}  \newcommand{\gru}{{\theta}}
\newcommand{\gri}{{\iota}} \newcommand{\grk}{{\kappa}} \newcommand{\grl}{{\lambda}} \newcommand{\grm}{{\mu}}
   \newcommand{\grp}{{\pi}}
 \newcommand{\grs}{{\sigma}}  
\newcommand{\grf}{{\phi}} \newcommand{\grx}{{\chi}} \newcommand{\grc}{{\psi}} \newcommand{\grv}{{\omega}}

\newcommand{\grF}{{\Phi}}   
\newcommand{\arw}{\rightarrow} 

\title[Decomposition matrices for the generic Hecke algebras on 3 strands]{Decomposition matrices for the generic Hecke algebras on 3 strands in  characteristic 0}
\author{Eirini Chavli}
\address{Institut f\"ur Algebra und Zahlentheorie, Universit\"at Stuttgart, Pfaffenwaldring 57, 
	70569 Stuttgart, Germany\\Office: 8.149\\
	Telephone: +49 (0)711 685 65518.}
\email{eirini.chavli@mathematik.uni-stuttgart.de}

\begin{document}
\maketitle
\begin{abstract}
	\noindent
	We determine all the decomposition matrices of the generic Hecke algebras on 3 strands in characteristic 0. These are the generic Hecke algebras associated with the exceptional complex reflection groups $G_4$, $G_8$, and $G_{16}$. We prove that for every choice of the parameters that define these algebras, all simple representations of the specialized algebra are obtained as modular reductions of simple representations of the generic algebra.
\end{abstract}
	
\section{Introduction}
	Between 1994 and 1998, M. Brou\'e, G. Malle, and R. Rouquier  
	generalized the definition of the Iwahori-Hecke algebra to the case of an arbitrary complex reflection group $W$ (see \cite{bmr}). This generalized algebra, which we denote by $H_W$, is known as the \emph{generic Hecke algebra}. It is defined over the Laurent polynomial ring $R:=\ZZ[u_1^{\pm},\dots, u_k^{\pm}]$, where $\{u_i\}_{1\leq i\leq k}$ is a set of parameters whose cardinality depends on $W$. In 1999, G. Malle proved that $H_W$ is split semisimple when defined over the field $\CC(v_1,\dots, v_k)$, where each parameter $v_i$ is an $N_i$-th root of $u_i$ for some specific $N_i\in \mathbb{N}$ (see \cite{malle1}, theorem 5.2). As a result, Tits' deformation theorem yields a bijection between the set of simple
	characters of $H_W$ and the set Irr($W$) of simple characters of $W$.
	
However, for many of the applications of the Hecke algebras it is important to know how the simple characters behave after specializing the parameters $u_i$ to arbitrary complex numbers. If the specialized Hecke algebra is semisimple, Tits' deformation theorem still applies; the simple characters of the specialized Hecke algebra are parametrized again by Irr$(W)$. However, this is not always the case. 
	
	If the specialized algebra is not semisimple, one needs to take a different approach.
	The simple characters of
	the semisimple Hecke algebra may not remain simple after specialization, however they are a linear combination of simple characters of the specialized algebra. 
		One can define the \emph{decomposition matrix}, which
	records the coefficients of this linear combination. 
		The rows and the columns of the decomposition matrix are indexed by simple characters; the rows by the ones of the generic algebra and the columns by the simple characters of the specialized algebra. This matrix offers a  depiction of the representation theory of the specialized algebra.

In 2011 M. Chlouveraki and H. Miyachi dealt with the cyclotomic Hecke algebras for $d$-Harish-Chandra series of rank 2 (see \cite{chlouveraki}). In this case, each algebra depends only on one parameter. Considering all values of the parameters for which the  algebras are not semisimple they managed to determine the decomposition matrices for these cases and they showed that they follow some specific matrix models. At this point, a number of questions arise: Are there any other matrix models  for the cyclotomic case outside the $d$-Harish-Chandra series that are not described by M. Chlouveraki and H. Miyachi? What happens in the generic case, where there is more than one parameter? Can one provide the decomposition matrix for the generic case without having a specific specialization?

Let $\gru: R\rightarrow \mathbb{C}$ be a specialization of $H_W$ and let
$n$ be the number of the simple characters of
the semisimple Hecke algebra and $m$ the number of the simple characters of the specialized algebra $H_W\otimes_{\gru}\mathbb{C}$. The main theorem of this paper is the following:
\begin{thm}
	Let $W$ be one of the exceptional groups $G_4$, $G_8$ and $G_{16}$. We have $m\leq n$. Moreover, one can reorder the rows of the decomposition matrix associated with $\gru$, so that it takes the following form:\newpage
	$$
	\tiny{\begin{blockarray}{cccccccc}
		\begin{block}{c[ccccccc]}
		&&&\\
		&&&\\
		&&\raisebox{-10pt}{{\huge\mbox{{$I_{m}$}}}}&\\[-3ex]
		&&&\\
		&&&\\
		&&&\\
		\cmidrule[0.00001cm](lr){2-7}
		&*&*&*\\
		&*&*&*\\
	&\vdots&\vdots&\vdots\\
	&*&*&*\\
		\end{block}
		\end{blockarray}}
	$$ 
	where $*$ denotes a placeholder for one of three values: 0, 1 or 2.
\end{thm}
The rows of the upper part of the matrix are indexed by a subset of Irr($W$), which is called by M. Chlouveraki and H. Miyachi  \emph{optimal basic set}   (see Definition 2 in \cite{chlouveraki}).
The appearance of the identity matrix implies that all simple representations of the specialized algebra are obtained as modular reductions of simple representations of the generic algebra. This result does not hold in general when we study modular representation theory (see, for example, \S 18.6 in \cite{serre2012linear} or Exercise 8.9 (e) in \cite{geck} as counterexamples). Moreover, the fact that $m<n$ (for the non-semisimple case) may have as explanation that the center of $H_W$ is \emph{reduction stable} (see Theorem 7.5.6 in \cite{geck}). However, the center of $H_W$ is not known yet, and such a condition could not be checked. \\\\
\textbf{Acknowledgments.} I would like to thank T. Conde,  M. Geck, S. K\"onig, J. K\"ulshammer and U. Thiel for fruitful discussions and references. Moreover, I would like to thank J. Michel for suggesting working with powers of the parameters in GAP, and M. Chlouveraki and G. Pfeiffer for carefully reading this paper.
\section{Preliminaries}
\subsection{Generic Hecke algebras on 3 strands}
A complex reflection group $W$ on a finite dimensional $\CC$-vector space $V$ is a finite subgroup of $GL(V)$, which is generated by pseudo-reflections, namely elements of finite order whose vector space of fixed points is a hyperplane.
We denote by $B_W$ the complex braid group associated with $W$, in the sense of Brou\'e-Malle-Rouquier (see \cite{bmr}).  A pseudo-reflection $s$ is called \emph{distinguished} if its only nontrivial eigenvalue on $V$ equals $\exp(-2\grp \gri /e_s)$, where  $\gri$ denotes a chosen imaginary unit, and $e_s$ the order of $s$ in $W$. To every distinguished pseudo-reflection $s$, one can associate homotopy classes in $B_W$ that are called \emph{braided reflections}  (for more details one may refer to \cite{bmr}). 

Let $R_W$ denote the Laurent polynomial ring $\ZZ[u_{s,i},u_{s,i}^{-1}]$, where $s$ runs over the set of the  distinguished pseudo-reflections of $W$, $1\leq i\leq e_s$, and  $u_{s,i}=u_{t,i}$ if $s$ and $t$ are conjugate in $W$. The \emph{generic Hecke algebra} $H_W$ associated with $W$ with parameters $(u_{s,1},\dots, u_{s,e_s})_s$ is the quotient of the group algebra $R_WB_W$ of $B_W$ by the ideal generated by the elements of the form
$$
(\grs-u_{s,1})(\grs-u_{s,2})\dots (\grs-u_{s,e_s}),$$
where $s$ runs over the conjugacy classes of the set of the  distinguished pseudo-reflections of $W$
 and $\grs$ over the set of braided reflections associated with $s$. It is enough to choose one such relation per conjugacy class, since  the  braided reflections associated with the same distinguished pseudo-reflection are conjugate in $B_W$.

 Let $B_3$ be usual braid group on 3 strands, with presentation given by generators the braids $s_1$ and $s_2$ and the (single)
 relation $s_1s_2s_1 = s_2s_1s_2$. We denote by $W_k$ the quotients of $B_3$ by the additional relation $s_i^k=1$, for $i=1,2$.
 Coxeter has shown that these quotients are finite if and only if $k\in\{2,3,4,5\}$ (see \textsection10 in \cite{Coxeter}). Apart from $k=2$,  which leads to the symmetric group $S_3$, we encounter for $k=3,4,5$ three exceptional complex reflection groups respectively, known as $G_4$, $G_8$, and $G_{16}$ in the Shephard-Todd classification (see \cite{shephard}). 
 We define as \emph{the generic Hecke algebras on 3 strands} the algebras $H_{W_k}$, $k=2,3,4,5$.
 \begin{rem}
 	By the relation $s_1s_2s_1 = s_2s_1s_2$ we have that $s_1$ and $s_2$ are conjugate in $W_k$ and, hence, the algebra $H_{W_k}$ is defined over the ring 
 	$R_{W_k}=\ZZ[u_1^{\pm1},...,u_k^{\pm1}]$. 
 \end{rem}
The algebra $H_{W_2}$ is an Iwahori-Hecke algebra and its modular representation theory has been studied in \cite{geck}. Hence, from
now on only the cases $k=3,4,5$ will be considered.
 The following theorem is Theorem 1.1 in \cite{chavli} and proves the \emph{BMR freeness conjecture} for the generic Hecke algebras on 3 strands. 
 	\begin{thm}$H_{W_k}$ is a free $R_{W_k}$-module of rank $|W_k|$.
 		\label{bmr}
 	\end{thm}
 	For the rest of this paper we make the following assumption, which has been conjectured in \cite{bmm} for every generic Hecke algebra associated with a complex reflection group.
 	\begin{as} $H_{W_k}$ has a unique symmetrizing trace $t_k: H_{W_k} \arw R_{W_k}$ having the properties described in \cite{bmm}, Theorem 2.1.
 		\label{sym}
 		\end{as}
 		 The assumption is known for $k=3$
 		 (see \cite{mallem} or  \cite{mw}). Moreover, after this work was completed, the author, together with  Boura, Chlouveraki, and Karvounis, provided another proof for $k=3$ and  proved the assumption for $k=4$ as well (see \cite{BCCK}).
 \subsection{Schur elements}
 \label{s}We denote by $\mathcal{K}_{W_k}$, $k=3,4,5$,  the splitting field of $W_k$,
 We denote by $\grm(\mathcal{K}_{W_k})$ the group of all roots of unity of $\mathcal{K}_{W_k}$ and, for every integer $m>1$, we set $\grz_m:=$exp$(2
 \grp \gri/m)$, where $\gri$  denotes a square root of -1.
 
 Let $\mathbf{v}=(v_1,...,v_k)$ be a set of $k$ indeterminates such that, for every $i\in\{1,...,k\}$, we have $v_i^{|\grm(\mathcal{K}_{W_k})|}=\grz_k^{-i}u_i$. By extension of scalars we obtain the algebra
 $\CC(\mathbf{v})H_{W_k}:=H_{W_k}\otimes_{R_{W_k}}\CC(\mathbf{v})$, which is split semisimple (see \cite{malle1}, theorem 5.2). Hence, by Tits' deformation theorem (see Theorem 7.4.6 in \cite{geck}) the specialization $v_i \mapsto 1$ induces a bijection
 Irr$(\CC(\mathbf{v})H_{W_k})\rightarrow$ Irr$(W_k)$, $\grx_k\mapsto \grx$.  By Theorem 7.2.6 in \cite{geck} we have:
 $$t_k=\sum\limits_{\grx\in\text{Irr}(W_k)}\frac{1}{s_{\grx}} \grx_k$$
 where $s_{\grx}$ denotes the \emph{Schur element} of $H_{W_k}$ associated with $\grx\in\text{Irr}(W_k)$, with respect to $t_k$.
 
Chlouveraki has shown that these elements are products of cyclotomic polynomials over $\mathcal{K}_{W_k}$, evaluated on monomials of degree 0 (see Τheorem 4.2.5 in \cite{chlouverakibook}).  One can refer to Michel's version of CHEVIE package of
GAP3 (see \cite{michelgap}) for this factorization. 
\begin{ex}
We consider the case of $G_4$, which  we denote in this paper as $W_3$. In CHEVIE the parameters must be in $Mvp$ form (which stands for \emph{multivariate polynomials}). We type:
\footnotesize{\begin{verbatim}
gap> W_3:=ComplexReflectionGroup(4);;
gap> CharNames(W_3);
[ "phi{1,0}", "phi{1,4}", "phi{1,8}", "phi{2,5}", "phi{2,3}", "phi{2,1}", "phi{3,2}"]
\end{verbatim}}\normalsize
\noindent
We see that $W_3$ admits 7 simple characters, known as $\grf_{i,j}$ in GAP notation, with $i$ denoting the degree and $j$ the fake degree of the representation. 
We find now the factorization of the Schur element of $H_{W_3}$ associated with the character $\phi_{1,4}$, which is the second element of the above list. We type:
\footnotesize{\begin{verbatim}
gap> S:=FactorizedSchurElements(Hecke(W_3,[[Mvp("u1"), Mvp("u2"), Mvp("u3")]]);;
gap> S[2];
-u1^-4u2^5u3^-1P2(u1u2^-2u3)P1P6(u1u2^-1)P1P6(u2u3^-1)
\end{verbatim}}\normalsize
\noindent
This means that 
$$s_{\phi_{1,4}} = -u_1^{-4}u_2^5u_3^{-1}\grF_2(u_1u_2^{-2}u_3)\grF_1(u_1u_2^{-1})\grF_6(u_1u_2^{-1})\grF_1(u_2u_3^{-1})\grF_6(u_2u_3^{-1}),$$
where $\grF_1$, $\grF_2$ and $\grF_6$ denote the cyclotomic polynomials $x-1$, $x+1$, and $x^2-x+1$, respectively.
\qed
\label{exg}
\end{ex}
\subsection{Decomposition matrix}
\label{sos}
Let $\gru: R_{W_k}\arw \mathbb{C}$ be a ring homomorphism, such that $\gru(u_i)\in \mathbb{C}^{\times}$, for every $i\in\{1,\dots,k\}$. We call $\gru$ a
\emph{specialization} of $R_{W_k}$. We set $\CC H_{W_k}:=H_{W_k}\otimes_{\gru}\CC$. This algebra is split, since it is a $\CC$-algebra (and, hence, assumption 7.4.1 (a) in \cite{geck} is satisfied). 

We suppose now that $\CC H_{W_k}$ is also semisimple. According to Theorem 7.5.11 in \cite{geck}, or to Lemma 2.6 in \cite{maller} this algebra is semisimple if and only if  $\gru(s_{\chi})\not=0$, for every $\chi \in\text{Irr}(\mathbb{C}(\textbf{v})H_{W_k})$.
 Using Tits' deformation theorem again, we obtain a canonical bijection between  the set of simple characters of $\CC H_{W_k}$ and  the set of simple characters of $\CC(\mathbf{v})H_{W_k}$,  which are in bijection with
the simple characters of $W_k$, as we mentioned in section \ref{s}. 

We now examine the
behavior of the simple representations of $\CC H_{W_k}$ in the non-semisimple case.
Let $R_0^{+}\big(\CC(\mathbf{v})H_{W_k}\big)$ (respectively $R_0^{+}(\CC H_{W_k})$) denote the subset of the \emph{Grothendieck group} of the category of finite dimensional $\CC(\mathbf{v})H_{W_k}$ (respectively $\CC H_{W_k})$-modules, which consists of elements $[V]$, where $V$ is a finite-dimensional $\CC(\mathbf{v})H_{W_k}$ (respectively $\CC H_{W_k})$-module, with relations $[V]=[V']+[V'']$, for each exact sequence $0\rightarrow V'\rightarrow V\rightarrow V'' \rightarrow 0$ (for more details, one may refer to \textsection 7.3 in \cite{geck}). By theorem 7.4.3 in \cite {geck} we obtain a well-defined decomposition map $$d^{\gru}: R_0^{+}\big(\CC(\mathbf{v})H_{W_k}\big) \arw R_0^{+}(\CC H_{W_k}).$$

The \emph{decomposition matrix} associated with the map $d^{\gru}$ is the
Irr$\big(\CC(\mathbf{v})H_{W_k}\big)\times$ Irr$(\CC H_{W_k})$ matrix $(d^{\gru}_{\grx\grf})$ with non-negative integer entries such that 
$$d^{\gru}([V_{\grx}])=\sum\limits_{\grf}d^{\gru}_{\grx\grf}[V'_{\grf}]\;\;\;\text{ for }\grx\in \text{Irr}(\CC(\mathbf{v})H_{W_k}),$$ where $V_{\chi}$ is a simple $\CC(\mathbf{v})H_{W_k}$-module of with character $\grx$ and $V'_{\grf}$  a simple $\CC H_{W_k}$-module of character $\grf$.
 
	We say that the characters $\grx, \grx' \in\text{Irr}(\CC(\mathbf{v})H_{W_k})$
 \emph{belong to the same block} if $\grx\not=\grx'$ and there is a chain of characters $\chi=\chi_1,\chi_2,\dots, \chi_n=\chi'$, where for every two neighbors $\grx_i$, $\grx_{i+1}$ there is a character $\grf_i\in\text{Irr}(\CC H_{W_k})$ such that $d^{\gru}_{\grx_i,\grf_i}\not=0\not=d^{\gru}_{\grx_{i+1},\grf_i}$.
	If the character $\grx\in\text{Irr}(\CC(\mathbf{v})H_{W_k})$  is alone in its block, then we call it a \emph{character of defect 0}.
	\subsection{Optimal basic sets}
	We recall that in the semisimple case, the simple representations of $\CC H_{W_k}$ are parametrized by the simple representations of $W_k$. The definition of \emph{optimal basic sets}, in the sense of Chlouveraki-Miyachi (see \cite{chlouveraki}), give ways to parametrize simple representations of $\CC H_{W_k}$
	in the non-semisimple case.
	\begin{defn}
		An optimal basic set $\mathcal{B}^{opt}$ for $\CC H_{W_k}$ with respect to $\gru$ is a subset of $\text{Irr}(W_k)$ such that the following two conditions are satisfied:
		\begin{enumerate}
			\item For all $\grf\in\text{Irr}(\CC H_{W_k})$ there exists $\grx_{\grf}\in\mathcal{B}^{opt}$ such that 
				\begin{enumerate}
					\item $d^{\gru}_{\grx_{\grf},\grf}=1$,
					\item If $d^{\gru}_{\grc,\grf}\not=0$ for some $\grc\in \text{Irr}(W_k)$, then either $\grc=\grx_{\grf}$ or $\grc\not \in \mathcal{B}^{opt}$.
					\end{enumerate}
					\item The map Irr$(\CC H_{W_k})\rightarrow\mathcal{B}^{opt}$, $\grf\mapsto \grx_{\grf}$ is a bijection.
		\end{enumerate}
	\end{defn}
	 Hence, $\mathcal{B}^{opt}$ is a set of characters of the
	generic algebra which remains irreducible upon specialization and has
	the further property that every other specialized character can be written
	as a sum of the specialized characters in $\mathcal{B}^{opt}$. If such a set exists, then  the decomposition matrix is 
	upper identity, with  the identity part consisting of the rows indexed by the elements of $\mathcal{B}^{opt}$. The existence of an optimal basic set implies that all simple modules are obtained after specialization of simple representations of the generic algebra.
	\begin{rem}
		If the algebra $\CC H_{W_k}$ is semisimple then $\mathcal{B}^{opt}=$Irr$(W_k)$.
	\end{rem}
	The rest of this paper is devoted to the proof of the existence of an optimal basic set for $\CC H_{W_k}$, $k=3,4,5$, in the non-semisimple case, with respect to any specialization $\gru$. 
	\section{The decomposition matrices of $W_k$}
	\subsection{Notation} Following the notation in GAP, we denote by $E(n)$, $n\in \NN$, the primitive $n$th root of unity $\exp(2\grp i/n)$. 
	\subsection{Methodology}\label{crrrrr} Motivated by the idea of Chlouveraki and Miyachi in \cite{chlouveraki} \textsection 3.1 we use the following criteria in order to calculate the decomposition matrix for $W_k$, $k=3,4,5$. We have also used some of these criteria in \cite{chavli}, in order to classify the simple representations of $B_3$ of small dimension.
	\begin{crr}
		Every 1-dimensional $\CC H_{W_k}$-module $V'_{\grf}$ is simple. Moreover, there is an 1-dimensional character $\grx\in\text{Irr}(\CC(\mathbf{v})H_{W_k})$, such that $d^{\gru}([V_{\grx}])=[V'_{\grf}]$.
		\label{c1}
	\end{crr}
	\begin{proof}
 By definition, the algebra $\CC H_{W_k}$ is the quotient of the group algebra $\CC B_3$ by the relations $(s_i-\gru(u_1))\dots (s_i-\gru(u_k))=0$, $i=1,2$. As a result,  
			the 1-dimensional $\CC H_{W_k}$-modules are of the form $s_1, s_2 \mapsto (\gru(u_j))$, $j=1,\dots,k$.
		\end{proof}
		\begin{crr}
		 2-dimensional modules are not simple if and only if they admit  1-dimensional submodules.
		 \label{c2}
		 \end{crr}
		We recall that $s_{\grx}$ denotes the Schur element associated with $\grx$. The next criterion summarized the results of Geck-Pfeiffer (see \cite{geck}, Theorem 7.5.11) and Malle-Rouquier (see \cite{maller}, Lemma 2.6).
		\begin{crr}
			$\gru(s_{\grx})\not =0$ if and only if $V_{\grx}$ is a simple module of defect 0. 
\label{c4}
\end{crr}
The next criterion follows directly from Lemma 7.5.10 in \cite{geck}.
\begin{crr}
 If $V_{\grx}, V_{\grc}$ are in  the same block, then $\gru(\grv_{\grx}(z_0))=\gru(\grv_{\grc}(z_0))$, where $\grv_{\grx}, \grv_{\grc}$ are the corresponding \emph{central characters} and $z_0$ is the central element $(s_1s_2)^3$.
 \label{c5}
 \end{crr}
 \begin{crr}
 Modular restrictions of the simple characters of $\CC(\mathbf{v}) H_{W_k}$
can be written uniquely as $\NN$-linear combinations of the simple characters of $\CC H_{W_k}$.	
\label{c6}	
	\end{crr}
	\begin{proof} Every $\CC(\mathbf{v}) H_{W_k}$-character can be written as $\NN$-linear combination of the simple characters of $\CC H_{W_k}$.	It remains to prove that the simple characters of $\CC H_{W_k}$ are linearly independent. Since the algebra $\CC H_{W_k}$ is split, the linear independence follows directly from Lemma 4.36 in \cite{meliot}.
		\end{proof}
		Notice that the above criteria can be used for any finite dimensional, symmetric algebra defined over a field. The following propositions are applied only to the generic Hecke algebras on 3 strands and give us a necessary and sufficient condition for a 2-dimensional, a 3-dimensional, and a 4-dimensional  $\CC H_{W_k}$-module to be simple. The last case refers only to $k\in\{3,4\}$, since there are not 4-dimensional simple $\CC H_{W_3}$-modules. Notice also that explicit matrix models for all simple representations are known (see \cite{brouem}) and are stored in the GAP3 package CHEVIE.
		
Let $a_i:=\gru(u_i)$, $i=1,\dots,k$, where $\gru$ is the specialization defined in \ref{sos}. Hence, by definition, $a_i\not=0$.
Let $V^m$, $m\in\{2,3,4\}$, be a simple $m$-dimensional $\CC(\mathbf{v})H_{W_k}$-module. We notice that the coefficients of the matrix models of each $V^m$ are Laurent polynomials with complex coefficients in variables 
$b_1,\dots, b_m$, where $b_i=a_{j_i}$ for some $j_i\in\{1,\dots,k\}$, such that $j_1<j_2<\dots<j_m$.
We denote by $U^m_{b_1,\dots, b_m}$  the $m$-dimensional $\CC H_{W_k}$-module, such that $d^{\gru}([V^m])=[U^m_{b_1,\dots,b_m}]$.  

\begin{prop}
	The  $\CC H_{W_k}$-module $U^2_{b_1,b_2}$ is simple if and only if $b_1^2-b_1b_2+b_2^2\not=0$.
	\label{2ir}
	\end{prop}
\begin{proof}
	We prove the case where $k=3$. The other cases can be proven similarly.  The matrix form of the  $\CC H_{W_k}$-module $U^2_{b_1,b_2}$ is the following:
	 \begin{equation*}
	 s_1\mapsto A:=\begin{pmatrix}
	 \phantom{-}b_1 & 0 \\
	 -b_1 & b_2
	 \end{pmatrix},  \;\;\;
	 s_2\mapsto  B:=\begin{pmatrix}
	 b_2 & b_2 \\
	 0 & b_1
	 \end{pmatrix}.
	 \end{equation*}
	 This module is not simple if and only if it  admits 1-dimensional submodules (criterion \ref{c2}). This is equivalent to the fact that the matrices $A$ and $B$ have a common eigenvector. The eigenvalues of $A$ are $b_1$ and $b_2$ with corresponding eigenvectors $v_{b_1}=\begin{pmatrix}
	b_1^{-1}(b_2-b_1)&1
	\end{pmatrix}^\intercal$ and $v_{b_2}=\begin{pmatrix}
	0 &
	1
	\end{pmatrix}^\intercal$. It is easy to check that $v_{b_2}$ is not an eigenvector for $B$. Moreover, we have $Bv_{b_1}=\begin{pmatrix}
	b_1^{-1}b_2^2&
	b_1
	\end{pmatrix}^\intercal=b_1\begin{pmatrix}
	b_1^{-2}b_2^2&
	1
	\end{pmatrix}^\intercal$, which means that $v_{b_1}$ is an eigenvector for $B$ if and only if $b_1^{-2}b_2^2=b_1^{-1}(b_2-b_1)$, which concludes the proof.
	\qedhere 
\end{proof}
The following corollary states that  an optimal basic set for $\mathbb{C}H_{W_k}$ with respect to any specialization $\gru$ contains always a 2-dimensional character.	
\begin{cor} For every specialization $\gru$ at least one $U^2_{b_1, b_2}$ is simple.
	\label{c2m}
\end{cor}
\begin{proof}
We first assume that the set $\{a_1,a_2,\dots,a_k\}$ has cardinality less than $k$. Without loss of generality, we may assume that $a_1=a_2$. Hence, the $\CC H_{W_k}$-module $U^2_{a_{1}, a_{2}}$ is simple, due to Proposition \ref{2ir}.
	
	Suppose now that the set $\{a_1,a_2,\dots,a_k\}$ has cardinality $k$ and that there is a specialization $\gru$ such that all the $\CC H_{W_k}$-modules $U^2_{b_1, b_2}$ are not simple.  Due to Proposition \ref{2ir} we have $b_1^2-b_1b_2+b_2^2=0$, for every $b_1,b_2\in\{a_1,a_2,\dots,a_k\}$.
	 More precisely, take
	$a_1^2-a_1a_2+a_2^2=a_1^2-a_1a_3+a_3^2=0$. Therefore, $a_3^2-a_2^2=a_1(a_3-a_2)$. Since $a_2\not=a_3$ we obtain $a_1=a_2+a_3$. Replacing now $a_1$ to the equality $a_1^2-a_1a_2+a_2^2=0$ we obtain $a_2^2+a_2a_3+a_3^2=0$, which contradicts the fact that $U^2_{a_2, a_3}$ is  simple.
\end{proof}
	\begin{prop}
			The $\CC H_{W_k}$-module $U^3_{b_1,b_2,b_3}$ is simple if and only if $$(b_1^2-b_2b_3)(b_2^2-b_1b_3)(b_3^2-b_1b_2)\not=0.$$
			\label{pr3d}
		\end{prop}
		\vspace*{-0.7cm}
		\begin{proof}
			In general, a 3-dimensional module is simple if and only if it does not admit 1-dimensional or 2-dimensional submodules. Let $s_1\mapsto A$ and $s_2\mapsto B$ the matrix form of the  $\CC H_{W_k}$-module $U^3_{b_1,b_2, b_3}$.
			 The existence of an 1-dimensional submodule translates into the existence of a common eigenvector for the matrices $A$ and $B$.  
			 
			 Let $DU^3_{b_1,b_2,b_3}:=$Hom$_{\CC}(U^3_{b_1,b_2,b_3},\CC)$ the $\CC H_{W_k}^{op}$-module, with action $(f*h)(u)=f(hu)$, for every $f\in DU^3_{b_1,b_2,b_3}$, $h\in\CC H_{W_k}$ and $u\in U^3_{b_1,b_2,b_3}$. Since $\CC H_{W_k}$ is a finite dimensional algebra defined over a field, the existence of a 2-dimensional $\CC H_{W_k}$- submodule (and, hence, the existence of an 1-dimensional $\CC H_{W_k}$- quotient) yields to the existence of an 1-dimensional $\CC H_{W_k}^{op}$- submodule. As a result, the transposed matrices 
			$A^\intercal$ and $B^\intercal$ must have a common eigenvector.
			
		Summing up, the $\CC H_{W_k}$-module $U^3_{b_1,b_2,b_3}$ is simple if and only if the matrices $A$ and $B$, on one hand, and the matrices	$A^\intercal$ and $B^\intercal$ on the other, do not have a common eigenvector. Since these matrices are stored in the GAP3 package CHEVIE, we can find the eigenvectors of the matrices $A$ and $A^\intercal$ and check under which conditions they are eigenvectors for the matrices  $B$ and $B^\intercal$, respectively. This is the method we used
	 for the proof of Proposition \ref{2ir}. One can check that the aforementioned condition is the one described in the hypothesis.
		\end{proof}	
		\begin{prop}Let $k\not=3$. The $\CC H_{W_k}$-module $U^4_{b_1,b_2,b_3,b_4}$ is simple if and only if $$(b_m^3-b_rb_lb_s)(b_m^2b_r^2+b_1b_2b_3b_4+b_l^2b_s^2)\not=0,$$
			for every $\{m,r,l,s\}=\{1,2,3,4\}$.
			\label{dim4}
		\end{prop}
		\begin{proof}
			In general, a 4-dimensional module is simple if and only if it does not admit 1-dimensional, 2-dimensional or 3-dimensional submodules. Let $s_1\mapsto A$ and $s_2\mapsto B$ the matrix form of the  $\CC H_{W_k}$-module $U^4_{b_1,b_2, b_3, b_4}$. The existence of an 1-dimensional submodule translates into the existence of a common eigenvector for the matrices $A$ and $B$. As in proof of Proposition \ref{pr3d}, the existence of a 3-dimensional submodule translates now into the existence of a common eigenvector for the transposed matrices 
			$A^\intercal$ and $B^\intercal$.
			Following the proof of Proposition \ref{2ir}, we conclude that there are 1-dimensional or 3-dimensional submodules if and only if $b_m^3-b_rb_lb_s=0$, with $m,r,l,s$ as in hypothesis.
			
			 Let now $S$ be a 2-dimensional $\CC H_{W_k}$- submodule of $U^4_{b_1,b_2,b_3,b_4}$. As  $\CC$-vector spaces, let $U^4_{b_1,b_2,b_3,b_4}=\langle c_1, c_2, c_3, c_4\rangle$ and $S=\langle d_1, d_2\rangle$. We write $d_1$ and $d_2$ as $\CC$-linear combinations of $c_1,\dots, c_4$ and we have $d_1=\sum\limits_{i=1}^4x_ic_i$ and $d_2=\sum\limits_{i=1}^4y_ic_i$. 
			Since $s_1d_1 \in S$, there are $\gra$, $\grb$ such that $
			\sum\limits_{i=1}^4x_i(s_1c_i)=\gra d_1 +\grb d_2$. Let $A:=(a_{ij})$. We have 
			$\sum\limits_{i=1}^4x_i(\sum\limits_{j=1}^4a_{ji}c_j)=\gra \sum\limits_{i=1}^4x_ic_i+\grb \sum\limits_{i=1}^4y_ic_i$. 
			We equalize the coefficients of $c_i$ and we have $\sum\limits_{j=1}^4a_{ij}x_j=\gra x_i+\grb y_i$, for every $i\in\{1,2,3,4\}$. Equivalently, 
			\begin{equation}
			\big(A-\gra I_4\big)\begin{pmatrix}
			x_1&x_2&x_3&x_4
			\end{pmatrix}^\intercal
			=\grb \begin{pmatrix}
			y_1&y_2&y_3&y_4
			\end{pmatrix}^\intercal.
			\label{coc}
			\end{equation}
			Similarly, since $s_1d_2\in S$, we have 
				\begin{equation}
				\big(A-\grd I_4\big) \begin{pmatrix}
				y_1&y_2&y_3&y_4
				\end{pmatrix}^\intercal=\grg \begin{pmatrix}
				x_1&x_2&x_3&x_4
				\end{pmatrix}^\intercal,
				\label{cocc}
				\end{equation}
				for some $\grg, \grd\in \CC$. 
		As a result, we have:
		$$	\big(A-\gra I_4\big)\big(A-\grd I_4\big)\begin{pmatrix}
		y_1&y_2&y_3&y_4
		\end{pmatrix}^\intercal=\grb\grg  \begin{pmatrix}
		y_1&y_2&y_3&y_4
		\end{pmatrix}^\intercal,$$
		meaning that $\begin{pmatrix}
			y_1&y_2&y_3&y_4
		\end{pmatrix}^\intercal$ is an eigenvector for the matrix $\big(A-\gra I_4\big)\big(A-\grd I_4\big)$. Similarly one can prove that there are $\gra'$, $\grd'\in \CC$ such that $\begin{pmatrix}
		y_1&y_2&y_3&y_4
		\end{pmatrix}^\intercal$ is an eigenvector for the matrix $\big(B-\gra' I_4\big)\big(B-\grd' I_4\big)$.
		
		Summing up, the $\CC H_{W_k}$-module $U^4_{b_1,b_2,b_3,b_4}$ admits a 2-dimensional submodule if and only if there are $\gra, \gra', \grd, \grd' \in \CC$ such as the matrices 
		$\big(A-\gra I_4\big)\big(A-\grd I_4\big)$ and $\big(B-\gra' I_4\big)\big(B-\grd' I_4\big)$ have a common eigenvector (not necessarily for the same eigenvalue). 
		Using Maple we proved that these matrices admit a common eigenvector if and only if 
		$b_m^2b_r^2+b_1b_2b_3b_4+b_l^2b_s^2=0$, with $m,r,l,s$ as in hypothesis.
			\end{proof}
\subsection{The case $\mathbf{k=3}$}
\label{s3333} The complex reflection group $W_3$ is denoted by $G_4$ in the Shephard-Todd classification and it admits the Coxeter-like presentation $\left\langle s_1, s_2\;|\;s_1^3=s_2^3=1, s_1s_2s_1=s_2s_1s_2\right\rangle$. The Hecke algebra $H_{W_3}$ is defined over the Laurent polynomial ring $R_{W_3}=\ZZ[u_1^{\pm}, u_2^{\pm}, u_3^{\pm}]$. We identify $s_i$ to their images in $H_{W_3}$ and the latter admits the presentation $$H_{W_3}=\left\langle s_1, s_2\;|\; s_1s_2s_1=s_2s_1s_2, \;(s_i-u_1)(s_i-u_2)(s_i-u_3)=0, \text{ for }  i=1,2\right\rangle.$$
We fix a specialization $\gru: R_{W_3}\rightarrow \CC$ of $R_{W_3}$, such that $u_1\mapsto a$, $u_2\mapsto b$, and $u_3\mapsto c$.

As we mentioned in Example \ref{exg}, the group $W_3$ admits 7 simple characters. More precisely, we have three 1-dimensional characters (the characters $\grf_{1,0}$,  $\grf_{1,4}$ and  $\grf_{1,8}$), three 2-dimensional characters (the characters $\grf_{2,5}$,  $\grf_{2,3}$ and  $\grf_{2,1}$) and one 3-dimensional character (the character $\grf_{3,2}$).
We now classify the decomposition matrices with respect to $\gru$ by distinguishing the following cases. Notice that this classification is up to permutation of the characters.
\begin{itemize}[leftmargin=-0.06cm]
\item[]\textbf{The set \boldmath{$\{a,b,c\}$} has cardinality 1:}
 Due to Criterion \ref{c1} and Propositions \ref{2ir} and \ref{pr3d} all characters are simple. Moreover, the 1-dimensional characters correspond to the same module, as well as the 2-dimensional ones. Hence, the decomposition matrix is the following:
\vspace*{-0.25cm}
\begin{center}
$\tiny{\begin{blockarray}[t]{cccc}
	\begin{block}{c[ccc]}
	\grf_{1,0}&&&\\
		\grf_{2,5}&&\raisebox{-7pt}{{\huge\mbox{{$I_{3}$}}}}&\\[-3ex]
		\grf_{3,2}&&&\\[0.5ex]
		\cmidrule[0.00001cm](lr){2-4}
	\grf_{1,4}&1&\cdot&\cdot\\
	\grf_{1,8}&1&\cdot&\cdot\\
	\grf_{2,3}&\cdot&1&\cdot\\
	\grf_{2,1}&\cdot&1&\cdot\\
		\end{block}
	\end{blockarray}}
	$
	\end{center}
	\vspace*{-0.2cm}
\item[]\textbf{The set \boldmath{$\{a,b,c\}$} has cardinality 2:} Up to permutation of the characters,  we may assume $a=b\not=c$.
The characters $\grf_{1,0}$ and $\grf_{1,4}$ correspond to the same module, as well as the characters $\grf_{2,5}$ and $\grf_{2,3}$. Moreover, due to Proposition \ref{2ir}, the character $\grf_{2,1}$ is simple.
 We distinguish the following cases, based on whether or not the character $\grf_{2,3}$  is simple:
 \begin{itemize}[leftmargin=0.65cm]
\item[\small{C1}.]$a^2-ac+c^2=0$: According to Proposition \ref{2ir} the character $\grf_{2,3}$ is not simple. We type:\vspace*{0.1cm}	\footnotesize{\begin{verbatim}
gap> T:=CharTable(Hecke(ComplexReflectionGroup(4),[[Mvp("a"), Mvp("b"), Mvp("c")]])).irreducibles;;
gap> t:=List(T,i->List(i,j->Value(j,["b", Mvp("a"), "c", -E(3)*Mvp("a")])));;
gap> t[5]=t[1]+t[3];
true
gap> s:=SchurElements(Hecke(ComplexReflectionGroup(4),[[Mvp("a"), Mvp("b"), Mvp("c")]]);;
gap> List(s,i->Value(i,["b", Mvp("a"), "c", -E(3)*Mvp("a")]));
gap> [0, 0, 0, 0, 0, E3, 3]
	\end{verbatim}}\normalsize
\noindent
	According to Criteria \ref{c4} and \ref{c6} the decomposition matrix is the following:
	\vspace*{-0.2cm}
	\begin{center}
	$
	\tiny{\begin{blockarray}[t]{ccccc}
	\begin{block}{c[cccc]}
	\grf_{1,0}&&&&\\
	\grf_{1,8}&&\;\;\;\;\raisebox{-10pt}{{\huge\mbox{{$I_{4}$}}}}&&\\[-3.9ex]
	\grf_{2,1}&&&&\\
	\grf_{3,2}&&&&\\
	\cmidrule[0.00001cm](lr){2-5}
	\grf_{1,4}&1&\cdot&\cdot&\cdot\\
	\grf_{2,5}&1&1&\cdot&\cdot\\
	\grf_{2,3}&1&1&\cdot&\cdot\\
	\end{block}
	\end{blockarray}}
	$
	\end{center}
	\vspace*{-0.2cm}
\item[\small{C2}.]$a^2-ac+c^2\not=0$: Due to Proposition \ref{2ir} the character $\grf_{2,5}$ is  simple and, hence,  it remains to investigate the behavior of $\grf_{3,2}$. Thanks to Proposition \ref{pr3d} it will be sufficient to examine three cases;
$a=-c$, $a^2=-c^2$, and $(a+c)(a^2+c^2)\not=0$. In the last case the character $\grf_{3,2}$ is of defect zero. For the other cases we apply criterion \ref{c6} as in Case C1. Therefore, we have the following decomposition matrices, one for each case respectively:
\vspace*{-0.1cm}
 \begin{center}
		 	$\tiny{\begin{blockarray}[t]{ccccc}
		 	\begin{block}{c[cccc]}
		 	\grf_{1,0}&&&&\\
		 	\grf_{1,8}&&\;\;\;\;\raisebox{-10pt}{{\huge\mbox{{$I_{4}$}}}}&&\\[-3.8ex]
		 		\grf_{2,5}&&&&\\
		 		\grf_{2,1}&&&&\\
		 			\cmidrule[0.00001cm](lr){2-5}
		 	\grf_{1,4}&1&\cdot&\cdot&\cdot\\
		 	\grf_{2,3}&\cdot&\cdot&1&\cdot\\
		 		\grf_{3,2}&1&\cdot&1&\cdot\\
		 	\end{block}
		 	\end{blockarray}}
		 	$\;\;\;\;\;\;\;\;
		 	$
		 	\tiny{\begin{blockarray}[t]{ccccc}
		 		\begin{block}{c[cccc]}
		 		\grf_{1,0}&&&&\\
		 		\grf_{1,8}&&\;\;\;\;\raisebox{-10pt}{{\huge\mbox{{$I_{4}$}}}}&&\\[-3.8ex]
		 		\grf_{2,5}&&&&\\
		 		\grf_{2,1}&&&&\\
		 		\cmidrule[0.00001cm](lr){2-5}
		 		\grf_{1,4}&1&\cdot&\cdot&\cdot\\
		 		\grf_{2,3}&\cdot&\cdot&1&\cdot\\
		 		\grf_{3,2}&\cdot&1&\cdot&1\\
		 		\end{block}
		 		\end{blockarray}}
		 	$\;\;\;\;\;\;\;\;
		 	$\tiny{\begin{blockarray}[t]{cccccc}
		  	\begin{block}{c[ccccc]}
		  	\grf_{1,0}&&&&&\\
		  	\grf_{1,8}&&&&&\\
		  	\grf_{2,5}&&&\raisebox{-10pt}{{\huge\mbox{{$I_{5}$}}}}&&\\[-3.8ex]
		  	\grf_{2,1}&&&&&\\
		  	\grf_{3,2}&&&&&\\
		  	\cmidrule[0.00001cm](lr){2-6}
		  	\grf_{1,4}&1&\cdot&\cdot&\cdot&\cdot\\
		  	\grf_{2,3}&\cdot&\cdot&1&\cdot&\cdot\\
		  	\end{block}
		  	\end{blockarray}}
		  $
		  \end{center}
\end{itemize}
\item[]\textbf{The set \boldmath{$\{a,b,c\}$} has cardinality 3:} We first notice that all characters correspond to distinct modules. According to Corollary \ref{c2m} at least one 2-dimensional character is simple. Based on Proposition \ref{2ir} and without loss of generality, we distinguish the following cases:
 \begin{itemize}[leftmargin=0.65cm]
\item[\small{C1}.] $a^2-ac+c^2=b^2-bc+c^2=0$: Since $a\not=b$ we may assume $a=-E(3)c$ and $b=-E(3)^2c$. 
Using Criterion \ref{c6} we obtain:
\vspace*{-0.6cm}
\begin{center}
	$\tiny{\begin{blockarray}[t]{cccccc}
	 		\begin{block}{c[ccccc]}
	 		\grf_{1,0}&&&&\\
	 		\grf_{1,4}&&\;\;\;\;\raisebox{-10pt}{{\huge\mbox{{$I_{4}$}}}}&&\\[-3.8ex]
	 		\grf_{1,8}&&&&\\
	 		\grf_{2,1}&&&&\\
	 		\cmidrule[0.00001cm](lr){2-5}
	 		\grf_{2,5}&\cdot&1&1&\cdot\\
	 		\grf_{2,3}&1&\cdot&1&\cdot\\
	\grf_{3,2}&1&1&1&\cdot\\
	 		\end{block}
	 		\end{blockarray}}
	$
	\end{center}
	\vspace*{-0.2cm}
 \item[\small{C2}.]$b^2-bc+c^2=0$ and $(a^2-ac+c^2)(a^2-ab+b^2)\not=0$: Due to Propositions \ref{2ir} and \ref{pr3d} the characters $\grf_{2,3}$, $\grf_{2,1}$ and $\grf_{3,2}$ are simple. 
   Using Criterion \ref{c6} again, we have:
   \vspace*{-0.2cm}
   \begin{center}
 $\tiny{\begin{blockarray}[t]{cccccccc}
 		\begin{block}{c[ccccccc]}
 		\grf_{1,0}&&&&&&\\
 		\grf_{1,4}&&&&&&\\
 		\grf_{1,8}&&&\;\raisebox{-10pt}{{\huge\mbox{{$I_{6}$}}}}&&&\\[-3.8ex]
 		\grf_{2,3}&&&&&&\\
 		\grf_{2,1}&&&&&&\\
 		\grf_{3,2}&&&&&&\\
 		\cmidrule[0.00001cm](lr){2-7}
 		\grf_{2,5}&\cdot&1&1&\cdot&\cdot&\cdot\\
 		\end{block}
 		\end{blockarray}}
 $
  \end{center}
  \vspace*{-0.2cm}
\item[\small{C3}.]$(b^2-bc+c^2)(a^2-ac+c^2)(a^2-ab+b^2)\not=0$: Due to Proposition \ref{2ir} all 2-dimensional characters are simple. If the character $\grf_{3,2}$ is simple, then we are in the semisimple case and the decomposition matrix is the identity matrix $I_7$. We suppose now that  $\grf_{3,2}$ is not simple.
According to Proposition \ref{pr3d} it will be sufficient to examine the case  $(a^2+bc)(b^2+ac)(c^2+ab)=0.$
Without loss of generality, we assume that $a^2+bc=0$. Using criterion \ref{c6} we obtain the following decomposition matrix:
\vspace*{-0.6cm}
\begin{center}
 $
 \tiny{\begin{blockarray}[t]{cccccccc}
 		\begin{block}{c[ccccccc]}
 		\grf_{1,0}&&&&&&\\
 		\grf_{1,4}&&&&&&\\
 		\grf_{1,8}&&&\raisebox{-10pt}{{\huge\mbox{{$I_{6}$}}}}&&&\\[-3.8ex]
 		\grf_{2,5}&&&&&&\\
 		\grf_{2,3}&&&&&&\\
 		\grf_{2,1}&&&&&&\\
 		\cmidrule[0.00001cm](lr){2-8}
 		\grf_{3,2}&1&\cdot&\cdot&1&\cdot&\cdot\\
 		\end{block}
 		\end{blockarray}}$
 	\end{center}
 	\vspace*{-0.4cm}
 		\end{itemize}
 	\end{itemize}
 	\subsection{The case $\mathbf{k=4}$} \label{4}The complex reflection group $W_4$ is denoted by $G_8$ in the Shephard-Todd classification and it admits the Coxeter-like presentation
 	$\left\langle s_1, s_2\;|\;s_1^4=s_2^4=1, s_1s_2s_1=s_2s_1s_2\right\rangle$. The Hecke algebra $H_{W_4}$ is defined over the ring $R_{W_4}=\ZZ[u_1^{\pm}, u_2^{\pm}, u_3^{\pm}, u_4^{\pm}]$. We identify again $s_i$ to their images in $H_{W_4}$ and the latter admits the following presentation:  $$H_{W_4}=\left\langle s_1, s_2\;|\; s_1s_2s_1=s_2s_1s_2, \;(s_i-u_1)(s_i-u_2)(s_i-u_3)(s_i-u_4)=0, \text{ for }  i=1,2\right\rangle.$$
 	We fix now a specialization $\gru: R_{W_4} \rightarrow \CC$ of $R_{W_4}$, such that $u_1\mapsto a$, $u_2\mapsto b$, $u_3\mapsto c$, and $u_4\mapsto d$.
 	
 The group $W_4$ admits 16 simple characters, which we simply denote by $\grf_i$, $i=1,\dots,16$ (we don't use the GAP notation). More precisely, 
 there are four characters of dimension 1 ($\grf_1,\dots, \grf_4$), six of dimension 2 ($\grf_5,\dots,\grf_{10}$), four of dimension 3 ($\grf_{11}, \dots, \grf_{14}$), and two of dimension 4 ($\grf_{15}$ and $\grf_{16}$).
We describe now the general methodology one can follow in order to find the decomposition matrix associated with every specialization.
\begin{itemize}[leftmargin=0.38cm]
	\item We first examine the characters $\grx$ with $\gru(s_{\grx})\not=0$. According to
	Criterion \ref{c4} these characters are of defect zero. 
\item We now examine the characters $\grx$ with $\gru(s_{\grx})=0$. There are two cases to be considered:
\vspace*{0.1cm}
\begin{itemize}[leftmargin=0.5cm]
	\item[$\triangleright$] The character $\grx$ is not simple, according to Propositions \ref{2ir}, \ref{pr3d}, or \ref{dim4}. Based on Criterion \ref{c6}, we use GAP in order to write $\grx$  as a linear combination of simple characters. Due to the appearance of square roots in the CHEVIE data of $H_{W_4}$, we use variables representing square roots of the parameters. Therefore, we prevent
	CHEVIE from
	extracting automatically these square roots, which may
	unavoidably be inconsistent with our expectations. We give the  example  $a^2-ab+b^2=0$:
	\vspace*{0.1cm}
	\footnotesize{\begin{verbatim}
		gap> H:=Hecke(ComplexReflectionGroup(8),[[Mvp("x")^2,Mvp("y")^2,Mvp("z")^2,Mvp("t")^2]]);;
		gap> T:=CharTable(H).irreducibles;;
		gap> t:=List(T,i->List(i,j->Value(j,["y", -E(12)^11*Mvp("x")])));;
		gap> t[5]=t[1]+t[2];
		true
		\end{verbatim}}\normalsize
	\noindent
	A different choice of the square root $y$ permutes the characters $\grf_{15}$ and $\grf_{16}$. This permutation depends on a specialization of the roots of
	the  parameters,  not  of  the  parameters themselves.
	\item[$\triangleright$]
	The character $\grx$ is simple, according to Criterion \ref{c1}, or to Propositions \ref{2ir}, \ref{pr3d}, \ref{dim4}. We check whether this character is in the same block with a character of the same dimension by using the  matrix models of the associated representation or Criterion \ref{c6}. If this does not happen, we check whether $\grx$ appears in the linear combination of a non-simple character, using the methodology we describe above for the example $a^2-ab+b^2=0$.
\end{itemize}
\end{itemize}

We notice here that the only case that needs further explanation is the case where the character $\grx$ is not simple, since one must be able to provide its linear combination of simple characters. Thanks to Propositions \ref{2ir}, \ref{pr3d} and \ref{dim4} we have necessary and sufficient condition for the simplicity of every character of dimension 2, 3 and 4. Combining them with Criterion \ref{c6} and using GAP, we are able to provide this linear combination. 

We recall at this point the case $k=3$, which we describe in Section \ref{s3333} and we notice that the non-simplicity of a character affects directly the behavior of the other characters. As a result, whenever one character is not simple, we are able to provide the exact decomposition matrix. However, for $k=4$ this does not always hold, due to the presence of 4 parameters.

The rest of this section is devoted to find the linear combination of the  non-simple characters, considering them by degree. When there is an intersection of cases, we give more details on the decomposition matrix, which sometimes are enough to provide it as whole. When the matrix is not provided as a whole, one  needs to repeat the above steps for the omitted characters and follow  the cases we present above.  
\subsubsection{The 2-dimensional characters}
\label{s2d}
Without loss of generality, we focus on $\grf_5$ and we assume that it is not a simple character.
Due to  Proposition \ref{2ir} we obtain
	 $a^2-ab+b^2=0$. According to the GAP example we give in the beginning of this section we have $\grf_5=\grf_1+\grf_2$. 
	We now examine the behavior of the rest of the characters. We first notice that
 $2\leq |\{a,b,c,d\}|\leq 4$, since $a\not=b$. We distinguish the following cases: 
\begin{itemize}[leftmargin=-0.01cm]
		\item[]\textbf{The set $\mathbf{\{a,b,c,d\}}$ has cardinality 2:} Since $a\not=b$ we have $\{a,b,c,d\}=\{a,b\}$. Up to permutation of the characters, it will be sufficient to examine the cases $\{b=c=d=-E(3)a\}$ and $\{b=-E(3)a,\; c=-E(3)a,\; d=a\}$. The decomposition matrices are respectively the following:
		\vspace{-0.3cm}
			\begin{center}
			$\fontsize{4}{4}\selectfont
			\begin{blockarray}[t]{ccccccccc}
					\begin{block}{c[cccccccc]}
						&&&&&&&\\
					&&&&\raisebox{-10pt}{{\huge\mbox{{$I_{7}$}}}}&&&\\[-3ex]
						&&&&&&&\\
						&&&&&&&\\
						\cmidrule[0.00001cm](lr){2-8}
							\grf_{3}&\cdot&1&\cdot&\cdot&\cdot&\cdot&\cdot\\
							\grf_{4}&\cdot&1&\cdot&\cdot&\cdot&\cdot&\cdot\\
							\grf_{5}&1&1&\cdot&\cdot&\cdot&\cdot&\cdot\\
							\grf_{6}&1&1&\cdot&\cdot&\cdot&\cdot&\cdot\\
							\grf_{7}&1&1&\cdot&\cdot&\cdot&\cdot&\cdot\\
							\grf_{9}&\cdot&\cdot&1&\cdot&\cdot&\cdot&\cdot\\
							\grf_{10}&\cdot&\cdot&1&\cdot&\cdot&\cdot&\cdot\\
							\grf_{13}&\cdot&\cdot&\cdot&1&\cdot&\cdot&\cdot\\
							\grf_{14}&\cdot&\cdot&\cdot&1&\cdot&\cdot&\cdot\\
					\end{block}
				\end{blockarray}
				$
				\;\;\;\;\;\;\;
$\fontsize{4}{4}\selectfont\begin{blockarray}[t]{ccccccccc}
	\begin{block}{c[cccccccc]}
	&&&&&&&\\
	&&&&\raisebox{-10pt}{{\huge\mbox{{$I_{7}$}}}}&&&\\[-3ex]
	&&&&&&&\\
	&&&&&&&\\
	\cmidrule[0.00001cm](lr){2-8}
	\grf_{3}&\cdot&1&\cdot&\cdot&\cdot&\cdot&\cdot\\
	\grf_{4}&1&\cdot&\cdot&\cdot&\cdot&\cdot&\cdot\\
	\grf_{5}&1&1&\cdot&\cdot&\cdot&\cdot&\cdot\\
	\grf_{6}&1&1&\cdot&\cdot&\cdot&\cdot&\cdot\\
	\grf_{9}&1&1&\cdot&\cdot&\cdot&\cdot&\cdot\\
	\grf_{10}&1&1&\cdot&\cdot&\cdot&\cdot&\cdot\\
	\grf_{13}&\cdot&\cdot&\cdot&\cdot&\cdot&1&\cdot\\
	\grf_{14}&\cdot&\cdot&\cdot&\cdot&1&\cdot&\cdot\\
	\grf_{15}&\cdot&\cdot&1&1&\cdot&\cdot&\cdot\\
	\end{block}
	\end{blockarray}
$
\end{center}
\vspace{-0.3cm}
The upper parts of the matrices are indexed by $\grf_1$, $\grf_2$, $\grf_8$, $\grf_{11}$, $\grf_{12}$, $\grf_{15}$, $\grf_{16}$ and
 $\grf_1$, $\grf_2$, $\grf_7$, $\grf_{8}$, $\grf_{11}$, $\grf_{12}$, $\grf_{16}$, respectively.
 \vspace*{0.2cm}
\item[]\textbf{The set $\mathbf{\{a,b,c,d\}}$ has cardinality 3}:
Since $a\not=b$ we assume $\{a,b,c,d\}=\{a,b,d\}$. Up to permutation of the characters, there are four cases to be considered; $\{b=c=-E(3)a,\; d=-E(3)^2a\}$,  $\{b=-E(3)a,\; c=d=-E(3)^2a\}$, $\{b=-E(3)a,\; c=a,\;d\not=-E(3)^2a\}$, and $\{b=-E(3)a,\; c=d\not=\pm E(3)^2a\}$. For the first two cases the decomposition matrices are respectively the following:
\vspace*{-0.6cm}
	\begin{center}
		$\fontsize{4}{4}\selectfont
		\begin{blockarray}[t]{cccccccccc}
		\begin{block}{c[ccccccccc]}
		&&&&&&&&&\\
		&&&&&\raisebox{-10pt}{{\huge\mbox{{$I_{9}$}}}}&&&&\\[-3ex]
		&&&&&&&&&\\
		&&&&&&&&&\\
		\cmidrule[0.00001cm](lr){2-10}
		\grf_{3}	&\cdot&1&\cdot&\cdot&\cdot&\cdot&\cdot&\cdot&\cdot\\
		\grf_{5}	&1&1&\cdot&\cdot&\cdot&\cdot&\cdot&\cdot&\cdot\\
		\grf_{6}	&1&1&\cdot&\cdot&\cdot&\cdot&\cdot&\cdot&\cdot\\
		\grf_{7}	&1&\cdot&1&\cdot&\cdot&\cdot&\cdot&\cdot&\cdot\\
		\grf_{10}	&\cdot&\cdot&\cdot&\cdot&1&\cdot&\cdot&\cdot&\cdot\\
		\grf_{12}	&1&1&1&\cdot&\cdot&\cdot&\cdot&\cdot&\cdot\\
		\grf_{13}	&1&1&1&\cdot&\cdot&\cdot&\cdot&\cdot&\cdot\\
		\end{block}
		\end{blockarray}
		$
		\;\;\;\;\;\;
$\fontsize{4}{4}\selectfont
\begin{blockarray}[t]{cccccccccc}
	\begin{block}{c[ccccccccc]}
	&&&&&&&&&\\
	&&&&&\raisebox{-10pt}{{\huge\mbox{{$I_{9}$}}}}&&&&\\[-3ex]
	&&&&&&&&&\\
	&&&&&&&&&\\
	\cmidrule[0.00001cm](lr){2-10}
	\grf_{4}	&\cdot&\cdot&1&\cdot&\cdot&\cdot&\cdot&\cdot&\cdot\\
	\grf_{5}	&1&1&\cdot&\cdot&\cdot&\cdot&\cdot&\cdot&\cdot\\
	\grf_{6}	&1&\cdot&1&\cdot&\cdot&\cdot&\cdot&\cdot&\cdot\\
	\grf_{7}	&1&\cdot&1&\cdot&\cdot&\cdot&\cdot&\cdot&\cdot\\
	\grf_{9}	&\cdot&\cdot&\cdot&1&\cdot&\cdot&\cdot&\cdot&\cdot\\
	\grf_{13}	&1&1&1&\cdot&\cdot&\cdot&\cdot&\cdot&\cdot\\
	\grf_{14}	&1&1&1&\cdot&\cdot&\cdot&\cdot&\cdot&\cdot\\
	\end{block}
	\end{blockarray}
$
\end{center}
\vspace{-0.2cm}
The upper parts of the matrices are indexed by  $\grf_1$, $\grf_2$, $\grf_4$, $\grf_8$, $\grf_9$, $\grf_{11}$, $\grf_{14}$, $\grf_{15}$, $\grf_{16}$ and $\grf_1$, $\grf_2$, $\grf_3$, $\grf_8$, $\grf_{10}$, $\grf_{11}$, $\grf_{12}$, $\grf_{15}$, $\grf_{16}$ respectively. 

We consider now the last two cases and we notice the presence of inequalities. As a result, we are not able to provide the decomposition matrix as a whole, since there are extra cases that need to be considered. We give these cases later in this paper, when we examine the behavour of 3 and 4-dimensional characters. The forms of these matrices are respectively the following:
\begin{center}
$\fontsize{4}{4}\selectfont
\begin{blockarray}[t]{ccccccc}
	\begin{block}{c[cccccc]}
	\grf_{1}&1&&&&&\\
	\grf_{2}&&1&&&&\\
	\grf_{4}&&&1&&&\\
	\grf_{7}&&&&1&&\\
	\grf_{8}&&&&&1&\\
	\grf_{9}&&&&&&1\\
	&&&&&&\\
	\grf_{3}&\cdot&1&\cdot&\cdot&\cdot&\cdot\\
	\grf_{5}&1&1&\cdot&\cdot&\cdot&\cdot\\
	\grf_{6}&1&1&\cdot&\cdot&\cdot&\cdot\\
	\grf_{10}&\cdot&\cdot&\cdot&\cdot&\cdot&1\\
	\end{block}
	\end{blockarray}$
	\;\;\;\;\;
$\fontsize{4}{4}\selectfont
\begin{blockarray}[t]{ccccccc}
		\begin{block}{c[cccccc]}
		\grf_{1}&1&&&&&\\
		\grf_{2}&&1&&&&\\
		\grf_{3}&&&1&&&\\
		\grf_{6}&&&&1&&\\
		\grf_{8}&&&&&1&\\
		\grf_{10}&&&&&&1\\
		&&&&&&\\
		\grf_{4}&\cdot&\cdot&1&\cdot&\cdot&\cdot\\
		\grf_{5}&1&1&\cdot&\cdot&\cdot&\cdot\\
		\grf_{6}&\cdot&\cdot&\cdot&1&\cdot&\cdot\\
		\grf_{9}&\cdot&\cdot&\cdot&\cdot&1&\cdot\\
		\end{block}
		\end{blockarray}$
		\end{center}
	\vspace*{-0.3cm}
\item[]\textbf{The cardinality of the set $\mathbf{\{a,b,c,d\}}$ is 4}: We first notice that the characters $\grf_6$, $\grf_7$, $\grf_8$, $\grf_9$ and $\grf_{10}$ correspond to distinct modules, since the associated matrix models are different. Moreover,
according to Proposition \ref{2ir} and the fact that $a,b,c,d$ are distinct, at least 3 of these characters are simple. Without loss of generality, we suppose that $\grf_6$, $\grf_7$, and $\grf_8$ are simple. We provide again the form of the decomposition matrix and not the decomposition matrix as a whole, since there are more cases to be considered, which we describe in detail in the following sections.
\vspace*{-0.3cm}
\begin{center}
	$\fontsize{4}{4}\selectfont
	\begin{blockarray}[t]{ccccccccccc}
		\begin{block}{c[cccccccccc]}
			&&&&&&&&&\\
		\grf_1&1&&&&&&&&\\
		\grf_2&&1&&&&&&&\\
		\grf_3&&&1&&&&&&\\
		\grf_4&&&&1&&&&&\\
		\grf_6&&&&&1&&&&\\
		\grf_7 &&&&&&1&&&\\
		\grf_8&&&&&&&1&&\\
	&&&&&&&&&\\
		\grf_9&\cdot&*&\cdot&*&\cdot&\cdot&\cdot&*\\
		\grf_{10}&\cdot&\cdot&*&*&\cdot&\cdot&\cdot&\cdot&*\\
		\grf_5&1&1&\cdot&\cdot&\cdot&\cdot&\cdot&\cdot&\cdot\\	
		\end{block}
		\end{blockarray}
	$
	\normalsize
\end{center}
\vspace*{-0.4cm}
where $*$ denotes a placeholder for one of two values, 0 or 1, depending on whether or not the characters $\grf_{9}$ and $\grf_{10}$ are simple.
The case where neither of these characters is simple is easy to be studied, due to Proposition \ref{2ir}. More precisely, since $a$ and $d$ are distinct and since $\grf_6$ is simple (i.e. $a^2-ac+c^2\not=0$) the characters $\grf_{9}$ and $\grf_{10}$ are not simple if and only if $d=E(3)^2a$ and $c=-a$. The decomposition matrix is the following: 
\vspace*{-0.3cm}
\begin{center}
	$\fontsize{4}{4}\selectfont
	\begin{blockarray}[t]{ccccccccccc}
		\begin{block}{c[cccccccccc]}
		&&&&&&&&&\\
		&&&&&\raisebox{-10pt}{{\huge\mbox{{$I_{9}$}}}}&&&&\\[-3ex]
		&&&&&&&&&\\
		&&&&&&&&&\\
		\cmidrule[0.00001cm](lr){2-11}
		\grf_{5}&1&1&\cdot&\cdot&\cdot&\cdot&\cdot&\cdot&\cdot\\
		\grf_{9}&\cdot&1&\cdot&1&\cdot&\cdot&\cdot&\cdot&\cdot\\
		\grf_{10}&\cdot&\cdot&1&1&\cdot&\cdot&\cdot&\cdot&\cdot\\
		\grf_{11}&\cdot&1&1&1&\cdot&\cdot&\cdot&\cdot&\cdot\\
		\grf_{13}&1&1&\cdot&1&\cdot&\cdot&\cdot&\cdot&\cdot\\
		\grf_{15}&1&1&1&1&\cdot&\cdot&\cdot&\cdot&\cdot\\
		\grf_{16}&\cdot&\cdot&\cdot&\cdot&\cdot&1&1&\cdot&\cdot\\
		\end{block}
		\end{blockarray}
	$
\end{center}
\vspace*{-0.3cm}
The upper part of the matrix is indexed by $\grf_1$, $\grf_2$, $\grf_3$, $\grf_4$, $\grf_6$, $\grf_7$, $\grf_8$, $\grf_{12}$, $\grf_{14}$.
\end{itemize}
\subsubsection{The 3-dimensional characters}
\label{s3d}
Without loss of generality, we focus on character  $\grf_{13}$ and we suppose it is not simple. According to Proposition \ref{pr3d} we have 
	 $(a^2+bd)(b^2+ad)(d^2+ab)=0$. Without loss of generality we assume that $d^2+ab=0$. 
	We type: 
	\footnotesize{\begin{verbatim}
		gap> t:=List(T,i->List(i,j->Value(j,["y", E(4)*t^2*Mvp("x")^-1)])));;
		gap> t[13]=t[4]+t[5];
		true
		\end{verbatim}}
	\normalsize
	\noindent
	At this point, we distinguish the following cases, depending on whether or not the character $\grf_5$ is simple (Proposition \ref{2ir}):
	\begin{itemize}[leftmargin=-0.003cm]
			\item[]	$\mathbf{a^2-ab+b^2=0}$: Since $b=-E(3)a$ and $d^2+ab=0$, we obtain $d=\pm E(3)^2a$. Hence, the characters $\grf_1$, $\grf_2$ and $\grf_4$ correspond to distinct modules, since $a,b$ and $d$ are distinct. Up to permutation of the characters, it will be sufficient to examine the case $d=-E(3)^2a$. The decomposition matrix is of the following form:
			\vspace*{-0.7cm}
			\begin{center}
				$\fontsize{4}{4}\selectfont
				\begin{blockarray}[t]{cccccccc}
					\begin{block}{c[ccccccc]}
					&&&&\\
					\grf_1&1&&&\\
					\grf_2&&1&&\\
					\grf_4&&&1&\\
					\grf_9&&&&1\\
					\grf_5&1&1&&\\
					\grf_7&1&&1&\\
					\grf_{13}&1&1&1&\\
					\end{block}
					\end{blockarray}
				$
				\end{center}
				\vspace*{-0.3cm}
		We now assume that there is another $3$-dimensional non-simple character. Without loss of generality, let $\grf_{14}$ be this character.
			 Due to Proposition \ref{pr3d}, it
			will be sufficient to  examine the cases $\{b=-E(3)a,\;c=d=-E(3)^2a\}$ and 
			$\{b=-E(3)a,\;c=E(3)^2a,\;d=-E(3)^2a\}$. The decomposition matrices are respectively the following:
	\vspace*{-0.3cm}\begin{center}
		 $\fontsize{4}{4}\selectfont
		  \begin{blockarray}[t]{ccccccccccc}
		 	\begin{block}{c[cccccccccc]}
		 	&&&&&&&&&\\
		 	&&&&&\raisebox{-10pt}{{\huge\mbox{{$I_{9}$}}}}&&&&\\[-3ex]
		 	&&&&&&&&&\\
		 	&&&&&&&&&\\
		 	\cmidrule[0.00001cm](lr){2-11}
		 		\grf_{4}&\cdot&\cdot&1&\cdot&\cdot&\cdot&\cdot&\cdot&\cdot\\
		 	\grf_{5}&1&1&\cdot&\cdot&\cdot&\cdot&\cdot&\cdot&\cdot\\
		 	\grf_{6}&1&\cdot&1&\cdot&\cdot&\cdot&\cdot&\cdot&\cdot\\
		 	\grf_{7}&1&\cdot&1&\cdot&\cdot&\cdot&\cdot&\cdot&\cdot\\
		 	\grf_{8}&\cdot&\cdot&\cdot&1&\cdot&\cdot&\cdot&\cdot&\cdot\\
		 	\grf_{13}&1&1&1&\cdot&\cdot&\cdot&\cdot&\cdot&\cdot\\
		 	\grf_{14}&1&1&1&\cdot&\cdot&\cdot&\cdot&\cdot&\cdot\\
		 	\end{block}
		 	\end{blockarray}
		 $
		 \;\;\;\;\;
		$
		\fontsize{4}{4}\selectfont\begin{blockarray}[t]{ccccccccccc}
			\begin{block}{c[cccccccccc]}
			&&&&&&&&&\\
			&&&&&\raisebox{-10pt}{{\huge\mbox{{$I_{9}$}}}}&&&&\\[-3ex]
		&&&&&&&&&\\
			&&&&&&&&&\\
			\cmidrule[0.00001cm](lr){2-11}
			\grf_{5}&1&1&\cdot&\cdot&\cdot&\cdot&\cdot&\cdot&\cdot\\
			\grf_{7}&1&\cdot&1&\cdot&\cdot&\cdot&\cdot&\cdot&\cdot\\
			\grf_{8} &\cdot&1&1&\cdot&\cdot&\cdot&\cdot&\cdot&\cdot\\
			\grf_{13}&1&1&\cdot&1&\cdot&\cdot&\cdot&\cdot&\cdot\\
			\grf_{14}&1&1&1&\cdot&\cdot&\cdot&\cdot&\cdot&\cdot\\
			\grf_{15}&\cdot&\cdot&\cdot&\cdot&1&1&\cdot&\cdot&\cdot\\
			\grf_{16}&1&1&1&1&\cdot&\cdot&\cdot&\cdot&\cdot\\
			\end{block}
			\end{blockarray}
		$
	\end{center}
	\vspace*{-0.3cm}
	The upper parts of the matrices are indexed by  $\grf_1$, $\grf_2$, $\grf_3$, $\grf_9$, $\grf_{10}$, $\grf_{11}$, $\grf_{12}$, $\grf_{15}$, $\grf_{16}$ and
		 $\grf_1$, $\grf_2$, $\grf_3$, $\grf_4$, $\grf_6$, $\grf_9$, $\grf_{10}$, $\grf_{11}$, $\grf_{12}$ respectively.
		
		Summing up, if a 3-dimensional character breaks up into three 1-dimensional characters and if another 3-dimensional character is not simple, then the latter  breaks up also into three 1-dimensional characters, while the rest of the 3-dimensional characters are of defect 0.\vspace*{0.1cm}
		\item[]		$\mathbf{a^2-ab+b^2\not=0}$: The character $\grf_{5}$ is simple and, since $\grf_{13}=\grf_4+\grf_5$, the decomposition matrix is of the following form:\vspace*{-0.8cm}\begin{center}
			$\tiny{\begin{blockarray}[t]{ccc}
				\begin{block}{c[cc]}
				\grf_4&1&\\
				\grf_5&&1\\
				\grf_{13}&1&1\\
				\end{block}
				\end{blockarray}}
			$
		\end{center}
		\vspace*{-0.6cm}
		\end{itemize}
\subsubsection{The 4-dimensional characters}
 We have two 4-dimensional characters, which are denoted as $\grf_{15}$ and $\grf_{16}$. We notice that $\gru(\grv_{\grf_{15}}(z_0))=-\gru(\grv_{\grf_{16}}(z_0))$. Therefore, the characters $\grf_{15}$ and $\grf_{16}$ are never in the same block (Criterion \ref{c5}) and hence, since they  have the maximal dimension, they are simple if and only if they are of defect 0.
  Without loss of generality, we examine the character $\grf_{15}$ and we suppose it is not a simple character. According to Proposition \ref{dim4} we distinguish the following cases:
\begin{itemize}[leftmargin=0.7cm]
	\item[\small{C1.}]$(a^3-bcd)(b^3-acd)(c^3-abd)(d^3-abc)=0$. Let $c^3-abd=0$. 
	We type: 
	\footnotesize{\begin{verbatim}
		gap> t:=List(T,i->List(i,j->Value(j,["y", Mvp("z")^3*Mvp("x")^-1*Mvp("t")^-1])));;
		gap> t[15]=t[3]+t[13];
		true
		\end{verbatim}}
	\normalsize
	\begin{itemize}[leftmargin=0.8cm]
		\item[\small{C1.1}]\normalsize	$(a^2+bd)(b^2+ad)(d^2+ab)=0$:
		Following Section \ref{s3d} we assume that $d^2+ab=0$. Since $c^3-abd=0$ it follows that $c^3=-d^3$. According to Section \ref{s3d}, we have:
		\vspace*{0.1cm}
		\begin{itemize}[leftmargin=-0.02cm]
			\item[] $\mathbf{a^2-ab+b^2\not=0}$: We have $b\not=E(3)^ka$, $k=1,2$. Therefore, from the equation $d^2+ab=0$ we obtain $d\not=\pm E(3)^ka$, $k=1,2$. Since $c^3=-d^3$ we obtain also $c\not=\pm E(3)^ka$, $k=1,2$.
			 According to Proposition \ref{2ir} the characters $\grf_{5}$, $\grf_6$ and $\grf_7$ are simple. The values of their central characters on $z_0$ are $-a^3b^3$, $-a^3c^3$, and $-a^3d^3$, respectively. Since $c^3\not=d^3$, the characters $\grf_6$ and $\grf_7$ are not in the same block (Criterion \ref{c4}). We distinguish now the following cases:
			 \begin{itemize}[leftmargin=0.4cm]
			\item[$\triangleright$] If $c=-d$, then the character $\grf_{10}$ is simple (Proposition \ref{2ir}). Since $c\not=d$ the characters $\grf_3$ and $\grf_4$ are distinct. Moreover, from the equation $d^2+ab=0$ we obtain $c^2+ab=0$. Hence, the character $\grf_{14}$ is not simple (see Section \ref{s3d}). Therefore, the decomposition matrix is of the following  form: \vspace*{-0.3cm}
			\begin{center}
			$\fontsize{4}{4}\selectfont
			\begin{blockarray}[t]{ccccccc}
					\begin{block}{c[cccccc]}
						&&&&&&\\
					\grf_3&1&&&&&\\
					\grf_4&&1&&&&\\
					\grf_5&&&1&&&\\
					\grf_6&&&*&*&&\\
					\grf_7&&&*&&*&\\
					\grf_{10}&&&*&*&*&*\\
					\grf_{13}&&1&1&&&\\ 
					\grf_{14}&1&&1&&&\\ 
					\grf_{15}&1&1&1&&&\\
					\end{block}
					\end{blockarray}
				$\end{center}
			\vspace*{-0.2cm}
				The $*$ denotes 
				a placeholder for one of two values, 0 or 1, whose sum in each line equals 1. Moreover, the characters $\grf_5$, $\grf_6$ and $\grf_7$ are not in the same block. 
				\item[$\triangleright$] If
			$c^2-cd+d^2=0$, then the character $\grf_{10}$ is not simple (see Proposition \ref{2ir}). Moreover, since $(a^2\pm ac+c^2)(a^2\pm ad+d^2)\not=0$  we have  $a\not =c$ and $a\not=d$. Hence, the characters $\grf_1$, $\grf_3$ and $\grf_4$ are distinct. Summing up, the decomposition matrix is of the form\vspace*{-0.2cm}
				\begin{center}
				$\fontsize{4}{4}\selectfont\begin{blockarray}[t]{cccccccc}
					\begin{block}{c[ccccccc]}
					&&&&&&\\
					\grf_3&1&&&&&\\
					\grf_4&&1&&&&\\
					\grf_1&&&1&&&&\\
					\grf_5&&&&1&&\\
					\grf_6&&&&*&*&\\
					\grf_7&&&&*&&*\\
					\grf_{10}&1&1&&&&\\
					\grf_{13}&&1&&1&&\\ 
					\grf_{15}&1&1&&1&&\\
					\end{block}
					\end{blockarray}$
				\end{center}
			\vspace*{-0.2cm}
				The $*$ denotes again
				a placeholder for one of two values, 0 or 1, whose sum in each line equals 1. The characters $\grf_5$, $\grf_6$ and $\grf_7$ cannot be  in the same block.
				\item[]  $\mathbf{a^2-ab+b^2=0}$: There are two  cases to be examined. First, we encounter the case $b=-E(3)^2a$, $c=-E(3)a$ and $d=E(3)a$, which has been examined (up to permutation of the characters) in \ref{s2d}. Secondly, we have the case $b=-E(3)^2a$, $c=-E(3)^2a$ and $d=E(3)a$:
				\vspace*{-0.7cm}
				\begin{center}$\fontsize{4}{4}\selectfont
					\begin{blockarray}[t]{ccccccccccc}
					\begin{block}{c[cccccccccc]}
					&&&&&&&&\\
					&&&&\raisebox{-10pt}{{\huge\mbox{{$I_{9}$}}}}&&&&\\[-3ex]
					&&&&&&&&\\
					&&&&&&&&\\
					\cmidrule[0.00001cm](lr){2-11}
					\grf_{3}&\cdot&1&\cdot&\cdot&\cdot&\cdot&\cdot&\cdot\\
					\grf_{5}&1&1&\cdot&\cdot&\cdot&\cdot&\cdot&\cdot\\
					\grf_{6}&1&1&\cdot&\cdot&\cdot&\cdot&\cdot&\cdot\\
					\grf_{9} &\cdot&1&\cdot&1&\cdot&\cdot&\cdot&\cdot\\
					\grf_{10} &\cdot&1&\cdot&1&\cdot&\cdot&\cdot&\cdot\\
					\grf_{12}&1&1&1&\cdot&\cdot&\cdot&\cdot&\cdot\\
					\grf_{13}&1&1&1&\cdot&\cdot&\cdot&\cdot&\cdot\\
					\grf_{15}&1&2&1&\cdot&\cdot&\cdot&\cdot&\cdot\\
					\end{block}
					\end{blockarray}
					$\end{center}
				\vspace*{-0.2cm}
				The upper part of the matrix is indexed by  $\grf_1$, $\grf_2$, $\grf_4$, $\grf_7$, $\grf_8$, $\grf_{11}$, $\grf_{14}$, $\grf_{16}$.
				\vspace*{0.1cm}
				\end{itemize}
				\end{itemize}
					\item[\small{C1.2}]\normalsize	$(a^2+bd)(b^2+ad)(d^2+ab)\not=0$:
					The character $\grf_{13}$ is simple. Therefore, we have the following form:
					\vspace*{-0.5cm}
					\begin{center}
						$\tiny{\begin{blockarray}[t]{ccc}
							\begin{block}{c[cc]}
							\grf_3&1&\\
							\grf_{13}&&1\\
							\grf_{15}&1&1\\
							\end{block}
							\end{blockarray}}
						$
					\end{center}
					\vspace*{-0.2cm}
					\end{itemize}
	\item[\small{C2.}] $(a^2b^2-abcd+c^2d^2)(a^2c^2-abcd+b^2d^2)(a^2d^2-abcd+b^2c^2)=0$. Let $ab=E(3)cd$. 
	We type: 
	\footnotesize{\begin{verbatim}
		gap> t:=List(T,i->List(i,j->Value(j,["t", E(6)*Mvp("x")*Mvp("y")*Mvp("z")^-1])));;
		gap> t[15]=t[5]+t[10];
		true
		\end{verbatim}}\normalsize
	\noindent
Since $ab\not=cd$ the characters $\grf_5$ and $\grf_{10}$ correspond to distinct modules, since the corresponding matrix models are different. We follow now Section \ref{s2d}. If $(a^2-ab+b^2)(c^2-cd+d^2)=0$, then we encounter (up to permutation of the characters) Case C1.1. If $(a^2-ab+b^2)(c^2-cd+d^2)\not=0$, then the decomposition matrix is of the following form:\vspace*{-0.3cm}\begin{center}
	$\tiny{\begin{blockarray}[t]{ccc}
			\begin{block}{c[cc]}
			\grf_5&1&\\
			\grf_{10}&&1\\
			\grf_{15}&1&1\\
			\end{block}
			\end{blockarray}}
		$
		\end{center}
		\vspace*{-0.55cm}
	\end{itemize}
	\subsection{The case $\mathbf{k=5}$} The complex reflection group $W_5$ is denoted by $G_{16}$ in the Shephard-Todd classification and it admits the Coxeter-like presentation $\left\langle s_1, s_2\;|\;s_1^5=s_2^5=1, s_1s_2s_1=s_2s_1s_2\right\rangle.$ The Hecke algebra $H_{W_5}$ is defined over the ring $R_{W_5}=\ZZ[u_1^{\pm}, u_2^{\pm}, u_3^{\pm}, u_4^{\pm}, u_5^{\pm}]$. We identify again $s_i$ to their images in $H_{W_5}$, and the latter admits the presentation $$H_{W_5}=\left\langle s_1, s_2\;|\; s_1s_2s_1=s_2s_1s_2, \;(s_i-u_1)(s_i-u_2)(s_i-u_3)(s_i-u_4)(s_i-u_5)=0, \text{ for }  i=1,2\right\rangle.$$
	We fix now a specialization $\gru: R_{W_5}\rightarrow \CC$ of $R_{W_5}$, such that $u_1\mapsto a$, $u_2\mapsto b$, $u_3\mapsto c$, $u_4\mapsto d$, and $u_5\mapsto e$.
	
	 We have 45 simple characters, which we denote here $\grf_i$, $i=1,\dots,45$.
	More precisely, we have five characters of dimension 1 ($\grf_1,\dots,\grf_5$), ten of dimension 2 ($\grf_6,\dots,\grf_{15}$), ten of dimension 3 ($\grf_{11}, \dots, \grf_{27}$), ten of dimension 4 ($\grf_{26},\dots, \grf_{35}$), five of dimension 5 ($\grf_{36},\dots,\grf_{40}$), and five of dimension 6 ($\grf_{41},\dots,\grf_{45}$).
We use the method we described in Section \ref{4} for $k=4$ and we consider again the characters by degree.
	\subsubsection{The 2-dimensional characters}
	\label{ddd2}
	Without loss of generality, we focus on  character $\grf_6$ and we suppose it is not simple. According to Proposition \ref{2ir} we have $a^2-ab+b^2=0$. We use this time variables representing 10th roots of the parameters and we type:
		\footnotesize{\begin{verbatim}
			gap> H:=Hecke(ComplexReflectionGroup(16),[[Mvp("x")^10,Mvp("y")^10,Mvp("z")^10,Mvp("t")^10,Mvp("w")^10]]);;
			gap> T:=CharTable(H).irreducibles;;
			gap> t:=List(T,i->List(i,j->Value(j,["y", E(12)^7*Mvp("x")])));;
			gap> t[6]=t[1]+t[2];
			true
			\end{verbatim}}\normalsize
		\noindent
		The decomposition matrix is of the following form:
		\vspace*{-0.3cm}
		\begin{center}
		$\tiny{
		\begin{blockarray}[t]{ccc}
			\begin{block}{c[cc]}
			\grf_1
			&1&\\
			\grf_2&&1\\
			\grf_6&1&1\\
			\end{block}
			\end{blockarray}}
			$
			\end{center}
			\vspace*{-0.4cm}
	\subsubsection{The 3-dimensional characters}
	\label{dd3}
Without loss of generality, we focus on the character $\grf_{16}$ and we assume it is not simple. According to Proposition \ref{pr3d} we have $(a^2+bc)(b^2+ac)(c^2+ab)=0$. Without loss of generality, we assume $c^2+ab=0$. 
	We type: 
	\footnotesize{\begin{verbatim}
		gap> t:=List(T,i->List(i,j->Value(j,["y", E(20)*Mvp("z")^2*Mvp("x")^-1])));;
		gap> t[16]=t[3]+t[6];
		true
		\end{verbatim}}\normalsize
	\noindent
	We distinguish now the following cases, depending on the behavior of $\grf_6$.
	\begin{itemize}[leftmargin=-0.005cm]
		\item[]	$\mathbf{a^2-ab+b^2\not=0}$: Due to Proposition \ref{2ir}, the character $\grf_{6}$ is simple. Moreover, the equation $c^2+ab=0$ together with the fact that $a^2-ab+b^2\not=0$ implies $c^2\pm ac+c^2\not=0$.  Furthermore, $c^2\pm cb+b^2\not=0$. Indeed, if $c^2\pm cb+b^2=0$ we would have $c=\pm E(3)b$ and, hence, due to the fact that $c^2=-ab$, we would obtain $a=-E(3)^2b$, which contradicts the hypothesis. Therefore, since $(c^2-ac+a^2)(c^2-cb+b^2)\not=0$, Proposition \ref{2ir} applies and, hence, the characters $\grf_7$ and $\grf_{10}$ are simple. 
		
		We examine now whether the characters $\grf_{6}$ and $\grf_{10}$ are in the same block or not. We have $\gru(\omega_{\grf_6}(z_0))=-a^3b^3$, $\gru(\omega_{\grf_7}(z_0))=-a^3c^3$ and $\gru(\omega_{\grf_{10}}(z_0))=-b^3c^3$. Notice that Criterion \ref{c4} provides a necessary but not sufficient condition for two characters being in the same block. As a result, we also use Criterion \ref{c6} in order to check if the aforementioned characters are indeed in the same block. We restrict ourselves on examining the cases
		$a=c=-b$,\quad $a=b=\pm E(4)c$, and $(c-a)(c-b)(a-b)\not=0$. The decomposition matrices have respectively the following forms:
		\vspace*{-0.2cm}
		\begin{center}
		$\fontsize{4}{4}\selectfont
		\begin{blockarray}[t]{ccccc}
				\begin{block}{c[cccc]}
				\grf_1&1&&&\\
				\grf_2&&1&&\\
					\grf_{6}&&&1&\\
						\grf_{7}&&&&1\\
				\grf_3&1&&&\\
				\grf_{10}&&&&1\\
				\grf_{16}&1&&&1\\
				\end{block}
				\end{blockarray}
			$\quad\quad
			 $\fontsize{4}{4}\selectfont
			 \begin{blockarray}[t]{ccccc}
				\begin{block}{c[cccc]}
				\grf_1&1&&&\\
				\grf_3&&1&&\\
				\grf_{6}&&&1&\\
				\grf_{7}&&&&1\\
					\grf_2&1&&&\\
				\grf_{10}&&&&1\\
				\grf_{16}&1&&1\\
				\end{block}
				\end{blockarray}
			$\quad\quad
		 $\fontsize{4}{4}\selectfont\begin{blockarray}[t]{ccccccc}
				\begin{block}{c[cccccc]}
				\grf_1&1&&&&&\\
				\grf_2&&1&&&&\\
				\grf_3&&&1&&&\\
				\grf_{6}&&&&1&&\\
				\grf_{7}&&&&&&1\\
				\grf_{10}&&&&&&1\\
				\grf_{16}&1&&&1&&\\
				\end{block}
				\end{blockarray}
			$
		\end{center}
		\vspace*{-0.2cm}
			\item[]	$\mathbf{a^2-ab+b^2=0}$: Using the same arguments as in section \ref{s3d}, the decomposition matrix is of the form:
			\vspace*{-0.6cm}
			\begin{center}
				$\fontsize{4}{4}\selectfont
				\begin{blockarray}[t]{ccccc}
				\begin{block}{c[cccc]}
				\grf_1&1&&&\\
				\grf_2&&1&&\\
				\grf_3&&&1&\\
				\grf_{10}&&&&1\\
				\grf_6&1&1&&\\
				\grf_7&1&&1&\\
				\grf_{16}&1&1&1&\\
				\end{block}
				\end{blockarray}
				$
			\end{center}
			\vspace*{-0.3cm}
			\end{itemize}
	\subsubsection{The 4-dimensional characters} 
	\label{sout}
	Without loss of generality, we assume that the character $\grf_{35}$ is not simple. We distinguish the following cases,
	based on Proposition \ref{dim4}.
\begin{itemize}[leftmargin=0.7cm]
\item[\small{C1.}]$(a^3-bcd)(b^3-acd)(c^3-abd)(d^3-abc)=0$: Without loss of generality, we assume that $d^3-abc=0$. 
We type: 
\footnotesize{\begin{verbatim}
	gap> t:=List(T,i->List(i,j->Value(j,["y", Mvp("t")^3*Mvp("x")^-1*Mvp("z")^-1])));;
	gap> t[35]=t[4]+t[16];
	true
	\end{verbatim}}
\normalsize
\noindent
We examine now the behavior of the character $\grf_{16}$.
\begin{itemize}[leftmargin=0.88cm]
	\item[\small{C1.1.}]\normalsize$(a^2+bc)(b^2+ac)(c^2+ab)\not=0$: Since $d^3-abc=0$ we obtain $(a^3+d^3)(b^3+d^3)(c^3+d^3)\not=0$.  Due to Propositions \ref{2ir} and \ref{pr3d} the characters $\grf_8$, $\grf_{11}$, $\grf_{13}$, and $\grf_{16}$ are simple. Hence, the decomposition matrix is of the following form:
	\vspace*{-0.3cm}
	\begin{center}
	$\tiny
	\begin{blockarray}[t]{cccccc}
		\begin{block}{c[ccccc]}
		\grf_4&1&&&&\\
		\grf_{8}&&1&&&\\
		\grf_{11}&&*&*&&\\
		\grf_{13}&&*&*&*&\\
		\grf_{16}&&&&&1\\
		\grf_{35}&1&&&&1\\
		\end{block}
		\end{blockarray}
	$\end{center}\vspace*{-0.2cm}
		where $*$ denotes 
		a placeholder for one of two values, 0 or 1, whose sum in each row equals 1.
	\item [\small{C1.2.}]\normalsize	$(a^2+bc)(b^2+ac)(c^2+ab)=0$:
	Let $c^2+ab=0$. Since $d^3=abc$ we obtain $c^3+d^3=0$. As we explained in Section \ref{dd3}, $\grf_{16}=\grf_3+\grf_6$. We distinguish the following cases (up to permutation of the characters), based on whether the character $\grf_6$ is simple or not.\vspace*{0.1cm}
	\begin{itemize}[leftmargin=-0.005cm]
	\item[]$\mathbf{a^2-ab+b^2=0}$: Up to permutation of the characters, there are exactly two cases to be examined; 
	$\{b=-E(3)a, c=-E(3)^2a, d=E(3)a\}$ and $\{b=-E(3)a, c=-E(3)^2a, d=a\}$. 
	The decomposition matrices are of the following forms, respectively:\vspace*{-0.2cm}
	\begin{center}
		$\fontsize{4}{4}\selectfont
		\begin{blockarray}[t]{ccccccccccc}
			\begin{block}{c[cccccccccc]}
			\grf_{1}&1&&&&&&&\\
			\grf_{2}&&1&&&&&&\\
			\grf_{3}&&&1&&&&&\\
			\grf_{4}&&&&1&&&&\\
			\grf_{8}&&&&&1&&&&\\
			\grf_{10}&&&&&&1&&&\\
			\grf_{11}&&&&&&&1&&\\
			\grf_{17}&&&&&&&&1&\\
				\grf_{22}&&&&&&&&&1\\
				&&&&&&&&&\\
			\grf_{6}&1&1&\cdot&\cdot&\cdot&\cdot&\cdot&\cdot&\cdot\\
			\grf_{7}&1&\cdot&1&\cdot&\cdot&\cdot&\cdot&\cdot&\cdot\\
			\grf_{13} &\cdot&\cdot&1&1&\cdot&\cdot&\cdot&\cdot&\cdot\\
			\grf_{16}&1&1&1&\cdot&\cdot&\cdot&\cdot&\cdot&\cdot\\
			\grf_{19}&1&\cdot&1&1&\cdot&\cdot&\cdot&\cdot&\cdot\\
			\grf_{30}&\cdot&\cdot&\cdot&\cdot&1&1&\cdot&\cdot&\cdot\\
			\grf_{35}&1&1&1&1&\cdot&\cdot&\cdot&\cdot&\cdot\\
			\end{block}
			\end{blockarray}
		$
	\;\;\;\;\;
	$\fontsize{4}{4}\selectfont
	\begin{blockarray}[t]{cccccccccc}
		\begin{block}{c[ccccccccc]}
		\grf_{1}&1&&&&&&\\
		\grf_{2}&&1&&&&&\\
		\grf_{3}&&&1&&&&\\
		\grf_{8}&&&&1&&&&\\
		\grf_{10}&&&&&1&&&\\
		\grf_{17}&&&&&&1&&\\
		\grf_{19}&&&&&&&1&\\
		\grf_{30}&&&&&&&&1\\
		&&&&&&&&\\
		\grf_{4}&1&\cdot&\cdot&\cdot&\cdot&\cdot&\cdot&\cdot\\
		\grf_{6}&1&1&\cdot&\cdot&\cdot&\cdot&\cdot&\cdot\\
		\grf_{7}&1&\cdot&1&\cdot&\cdot&\cdot&\cdot&\cdot\\
			\grf_{11}&1&1&\cdot&\cdot&\cdot&\cdot&\cdot&\cdot\\
		\grf_{13} &1&\cdot&1&\cdot&\cdot&\cdot&\cdot&\cdot\\
		\grf_{16}&1&1&1&\cdot&\cdot&\cdot&\cdot&\cdot\\
		\grf_{22}&1&1&1&\cdot&\cdot&\cdot&\cdot&\cdot\\
		\grf_{35}&2&1&1&\cdot&\cdot&\cdot&\cdot&\cdot\\
		\end{block}
		\end{blockarray}
	$\end{center}
\vspace*{-0.2cm}
	For the second decomposition matrix we also have $\grf_{18}=\grf_{24}$, $\grf_{20}=\grf_{25}$, $\grf_{26}=\grf_{34}$, $\grf_{29}=\grf_{31}$, and $\grf_{41}=\grf_{44}$.
\vspace*{0.1cm}
\item[] $\mathbf{a^2-ab+b^2\not=0}$: Since $c^3+d^3=0$ and $c^2+ab=0$, we obtain $(d^3\pm a^3)(c^3\pm a^3)\not=0$. Due to Proposition \ref{2ir}, the characters $\grf_{6}$, $\grf_7$, $\grf_8$, $\grf_{10}$ and $\grf_{11}$ are simple. Moreover, the values of their central characters on $z_0$ are $-a^3b^3$,  $-a^3c^3$,  $-a^3d^3$, $-c^3b^3$, and $-d^3b^3$, respectively.
As a result, the characters $\grf_{6}$, $\grf_{7}$ and $\grf_{8}$ are distinct, as well as the characters $\grf_{10}$ and $\grf_{11}$ (Criterion \ref{c5}). Hence:
\vspace*{0.1cm}
\begin{center}\fontsize{4}{4}\selectfont
	$\begin{blockarray}[t]{ccccccccc}
	\begin{block}{c[cccccccc]}
	\grf_1&1&&&&&&&\\
	\grf_{3}&&1&&&&&&\\
	\grf_{4}&&&1&&&&&\\
	\grf_{6}&&&&1&&&&\\
	\grf_{7}&&&&&1&&&\\
	\grf_{8}&&&&&&1&&\\
	\grf_{10}&&&&&*&*&*&\\
	\grf_{11}&&&&&*&*&&*\\	
	\grf_{16}&&1&&1&&&\\
	\grf_{35}&&1&1&1&&&\\
	&&&&&&&&\\
	\end{block}
	\end{blockarray}
	$\end{center}
\vspace*{-0.2cm}
The $*$ denotes 
a placeholder for one of two values, 0 or 1, whose sum in each line and in each column equals 1.\vspace*{0.1cm}
		\end{itemize}
	\end{itemize}
	\item[\small{C2.}] $(a^2b^2-abcd+c^2d^2)(a^2c^2-abcd+b^2d^2)(a^2d^2-abcd+b^2c^2)=0$: Let $ab=E(3)cd$. 
	We type: 
	\footnotesize{\begin{verbatim}
		gap> t:=List(T,i->List(i,j->Value(j,["t", E(6)*Mvp("x")*Mvp("y")*Mvp("z")^-1])));;
		gap> t[35]=t[6]+t[13];
		true
		\end{verbatim}}\normalsize
	\noindent
	Since $ab\not=cd$ the matrix models associated with $\grf_6$ and $\grf_{13}$ are different. 	If $(a^2-ab+b^2)(c^2-cd+d^2)=0$, we fall into Case C1.2. We assume now $(a^2-ab+b^2)(c^2-cd+d^2)\not=0$. The decomposition matrix is of the form:
	\vspace*{-0.68cm}\begin{center}
		$\tiny{\begin{blockarray}[t]{ccc}
			\begin{block}{c[cc]}
			\grf_6&1&\\
			\grf_{13}&&1\\
			\grf_{35}&1&1\\
			\end{block}
			\end{blockarray}}$
	\end{center}
	\vspace*{-0.4cm}
\end{itemize}
\subsubsection{The 5-dimensional characters}
\label{dimm5}
For $i,j\in\{36,37,\dots,40\}$ with $i\not=j$ we notice that
$\gru(\grv_{\grf_{i}}(z_0))=
-\gru(\grv_{\grf_{j}}(z_0))$. Therefore, according to Criterion \ref{c5}, different 5-dimensional characters are not in the same block. Moreover, the matrix models of the corresponding modules depend on $a,b,c,d,e$ and on the choice of a 5th root of $abcde$, which we denote by $r$. Let $U^r$ be the corresponding $\CC H_5$-module. We have the following result:
\begin{prop}The $\CC H_5$-module $U^r$ is simple if and only if 
	$(r^2+\gra r+\gra^2)(r^2+\gra\grb)\not=0,$
	for every $\gra, \grb \in \{a,b,c,d\}$ with $\gra\not=\grb$.
	\label{d5}
\end{prop}
\begin{proof}
	A 5-dimensional module is simple if and only if it doesn't admit 1, 2, 3, or 4-dimensional submodules. Let $s_1\mapsto A$ and $s_2\mapsto B$ the matrix form of the  $\CC H_5$-module $U^r$.
	Following the proof of Proposition \ref{dim4},
	the existence of an 1-dimensional and 4-dimensional submodule translates into the existence of a common eigenvector for the matrices $A$ and $B$ and a common eigenvector for the matrices 
	$A^\intercal$ and $B^\intercal$. 
	As in proof of Proposition \ref{2ir}, we conclude that there are 1-dimensional and 4-dimensional submodules if and only if $r^2+\gra r+\gra^2=0$, for every $\gra \in\{a,b,c,d,e\}$.
	
	It remains to exclude the existence of 2 and 3-dimensional submodules. 
Following the proof of Proposition \ref{dim4}, the
 $\CC H_5$-module $U^r$ admits a 2-dimensional submodule if and only if there are $\grl_1, \grl_1', \grl_2, \grl_2' \in \CC$ such that the matrices 
	$\big(A-\grl_1 I_4\big)\big(A-\grl_2 I_4\big)$ and $\big(B-\grl_1' I_4\big)\big(B-\grl_2' I_4\big)$ have a common eigenvector (not necessarily for the same eigenvalue).
	The existence of a 3-dimensional submodule and, hence, a 3-dimensional quotient, translates into the existence of a 2-dimensional $\CC H_5^{op}$-submodule. As a result, there are $\grm_1, \grm_1', \grm_2, \grm_2' \in \CC$ such that the matrices 
	$\big(A^\intercal-\grm_1 I_4\big)\big(A^\intercal-\grm_2 I_4\big)$ and $\big(B^\intercal-\grm_1' I_4\big)\big(B^\intercal-\grm_2' I_4\big)$ have a common eigenvector (not necessarily for the same eigenvalue).
	Using Maple we proved that there are  common eigenvectors for the aforementioned  matrices  if and only if 
	$r^2+\gra\grb=0$, for every $\gra, \grb \in\{a,b,c,d\}$, $\gra\not=\grb$.
\end{proof}
Without loss of generality, we assume that the character $\grf_{40}$ is not simple.
We distinguish the following cases, based on Proposition \ref{d5}:
\begin{itemize}	[leftmargin=*]
	\item[\small{C1.}]$r^2+\gra r+\gra^2=0$, for some $\gra\in\{a,b,c,d,e\}$. Let $r^2+er+e^2=0$. Hence, $r=E(3)e\Rightarrow abcde=E(3)^2e^5\Rightarrow e^4=E(3)abcd$. 
	 We type: 
	 \footnotesize{\begin{verbatim}
	 	gap> t:=List(T,i->List(i,j->Value(j,["t", E(3)^2*Mvp("x")^-1*Mvp("y")^-1*Mvp("z")^-1*Mvp("w")^4])));;
	 	gap> t[40]=t[5]+t[35];
	 	true
	 	\end{verbatim}}\normalsize
\noindent
	  In order to have $\grf_{40}=\grf_5+\grf_{35}$ and not 
	$\grf_{40}=\grf_5+\grf_{30}$, we must choose  $t=E(3)x^{-1}y^{-1}z^{-1}w^4$ instead of $t=-E(3)x^{-1}y^{-1}z^{-1}w^4$. 	We follow now Section \ref{sout}. Since we work with 10th roots of the parameters, one may notice that there are less choices for the parameter $e$.
	\vspace*{0.1cm}
	\begin{itemize}[leftmargin=0.9cm]
			\item[\small{C1.1.}]\small{$(a^3-bcd)(b^3-acd)(c^3-abd)(d^3-abc)(a^2b^2-r^5+c^2d^2)(a^2c^2-r^5+b^2d^2)(a^2d^2-r^5+b^2c^2)\not=0$}: \normalsize \vspace*{-0.3cm}
			\begin{center}
				$
			\fontsize{4}{4}\selectfont
				\begin{blockarray}[t]{cccc}
					\begin{block}{c[ccc]}
					\grf_{5}&1&&\\
					\grf_{30}&&1&\\
					\grf_{35}&&&1\\
					\grf_{40}&1&&1\\
					\end{block}
					\end{blockarray}
				$
			\end{center}
			\vspace*{-0.2cm}
		\item[\small{C1.2.}]	$e^4=E(3)d^4$ and $(a^2+bc)(b^2+ac)(c^2+ab)\not=0$: Since $e\not=d$, the characters $\grf_4$ and $\grf_5$ are distinct. Following C1.1 of Section \ref{sout},  the decomposition matrix is of the form
		\vspace*{-0.2cm}
		\begin{center}
			$\fontsize{4}{4}\selectfont
			\begin{blockarray}[t]{ccccccc}
				\begin{block}{c[cccccc]}
				\grf_4&1&&&&&\\
				\grf_5&&1&&&&\\
				\grf_{8}&&&1&&&\\
				\grf_{11}&&&*&*&&\\
				\grf_{13}&&&*&*&*&\\
				\grf_{16}&&&&&&1\\
				\grf_{35}&1&&&&&1\\
				\grf_{40}&1&1&&&&1\\
				\end{block}
				\end{blockarray}
			$\end{center}
		\vspace*{-0.3cm}
		The $*$ denotes 
		a placeholder for one of two values, 0 or 1, whose sum in each line is 1.
		\item [\small{C1.3.}]	$b=-E(3)a$, $c=-E(3)^2a$, $d=E(3)a$, $e=a$:\vspace*{0.1cm} \begin{center}
			$\fontsize{4}{4}\selectfont
		\begin{blockarray}{ccccccccccccc}
				\begin{block}{c[cccccccccccc]}
				&&&&&\raisebox{-10pt}{{\huge\mbox{{$I_{25}$}}}}&&&&&&\\[-3ex]
				&&&&&&&&&&&\\
				&&&&&&&&&&&\\
				\cmidrule[0.00001cm](lr){2-13}
				\grf_{5}&1&\cdot&\cdot&\cdot&\cdot&\cdot&\cdot&\cdot&\cdot&\cdot&\cdot\\
				\grf_{6}&1&1&\cdot&\cdot&\cdot&\cdot&\cdot&\cdot&\cdot&\cdot&\cdot\\
				\grf_{7}&1&\cdot&1&\cdot&\cdot&\cdot&\cdot&\cdot&\cdot&\cdot&\cdot\\
				\grf_{12}&1&1&\cdot&\cdot&\cdot&\cdot&\cdot&\cdot&\cdot&\cdot&\cdot\\
				\grf_{13} &\cdot&\cdot&1&1&\cdot&\cdot&\cdot&\cdot&\cdot&\cdot&\cdot\\
				\grf_{14}&1&\cdot&1&\cdot&\cdot&\cdot&\cdot&\cdot&\cdot&\cdot&\cdot\\
				\grf_{16}&1&1&1&\cdot&\cdot&\cdot&\cdot&\cdot&\cdot&\cdot&\cdot\\
				\grf_{19}&1&\cdot&1&1&\cdot&\cdot&\cdot&\cdot&\cdot&\cdot&\cdot\\
				\grf_{23}&1&1&1&\cdot&\cdot&\cdot&\cdot&\cdot&\cdot&\cdot&\cdot\\
				\grf_{24}&\cdot&\cdot&\cdot&\cdot&\cdot&\cdot&\cdot&1&\cdot&\cdot&\cdot\\
				\grf_{25}&1&\cdot&1&1&\cdot&\cdot&\cdot&\cdot&\cdot&\cdot&\cdot\\
				\grf_{26}&\cdot&\cdot&\cdot&\cdot&\cdot&1&1&\cdot&\cdot&\cdot&\cdot\\
				\grf_{30}&\cdot&\cdot&\cdot&\cdot&1&1&\cdot&\cdot&\cdot&\cdot&\cdot\\
				\grf_{31}&1&1&1&1&\cdot&\cdot&\cdot&\cdot&\cdot&\cdot&\cdot\\
				\grf_{35}&1&1&1&1&\cdot&\cdot&\cdot&\cdot&\cdot&\cdot&\cdot\\
				\grf_{40}&2&1&1&1&\cdot&\cdot&\cdot&\cdot&\cdot&\cdot&\cdot\\
				\grf_{42}&\cdot&\cdot&\cdot&\cdot&\cdot&\cdot&\cdot&\cdot&1&1&\cdot\\
				\grf_{43}&2&1&2&1&\cdot&\cdot&\cdot&\cdot&\cdot&\cdot&\cdot\\
				\grf_{45}&\cdot&\cdot&\cdot&\cdot&\cdot&\cdot&\cdot&\cdot&\cdot&\cdot&1\\
				\end{block}
				\end{blockarray}
			$\end{center}
		\vspace*{-0.2cm}
		The upper part of the matrix is labeled by $\grf_1$, $\grf_2$, $\grf_3$, $\grf_4$, $\grf_8$, $\grf_{10}$, $\grf_{15}$,   $\grf_{17}$,  
		$\grf_{18}$,   $\grf_{22}$, and   $\grf_{41}$ together with the  characters  $\grf_9$, $\grf_{11}$, $\grf_{20}$,   $\grf_{21}$,  
		$\grf_{27}$,   $\grf_{28}$, $\grf_{29}$, $\grf_{32}$, $\grf_{33}$, $\grf_{34}$, $\grf_{36}$, $\grf_{37}$, $\grf_{38}$, $\grf_{39}$, and  $\grf_{44}$, which are of defect 0.	
			\item[\small{C1.4.}]$b=-E(3)a$, $c=-E(3)^2a$, $d=a$, $e=E(3)a$: 
			\vspace*{-0.2cm}\begin{center}
				$\fontsize{4}{4}\selectfont
				\begin{blockarray}[t]{cccccccccccc}
				\begin{block}{c[ccccccccccc]}
				&&&&&\raisebox{-10pt}{{\huge\mbox{{$I_{24}$}}}}&&&&&\\[-3ex]
				&&&&&&&&&&\\
				&&&&&&&&&&\\
				\cmidrule[0.00001cm](lr){2-12}
				\grf_{4}&1&\cdot&\cdot&\cdot&\cdot&\cdot&\cdot&\cdot&\cdot&\cdot\\
				\grf_{6}&1&1&\cdot&\cdot&\cdot&\cdot&\cdot&\cdot&\cdot&\cdot\\
				\grf_{7}&1&\cdot&1&\cdot&\cdot&\cdot&\cdot&\cdot&\cdot&\cdot\\
				\grf_{11}&1&1&\cdot&\cdot&\cdot&\cdot&\cdot&\cdot&\cdot&\cdot\\
				\grf_{13}&1&\cdot&1&\cdot&\cdot&\cdot&\cdot&\cdot&\cdot&\cdot\\
				\grf_{14}&1&\cdot&\cdot&1&\cdot&\cdot&\cdot&\cdot&\cdot&\cdot\\
				\grf_{15}&\cdot&\cdot&\cdot&\cdot&1&\cdot&\cdot&\cdot&\cdot&\cdot\\
				\grf_{16}&1&1&1&\cdot&\cdot&\cdot&\cdot&\cdot&\cdot&\cdot\\
				\grf_{20}&1&\cdot&1&1&\cdot&\cdot&\cdot&\cdot&\cdot&\cdot\\
				\grf_{22}&1&1&1&\cdot&\cdot&\cdot&\cdot&\cdot&\cdot&\cdot\\
				\grf_{24}&\cdot&\cdot&\cdot&\cdot&\cdot&\cdot&\cdot&1&\cdot&\cdot\\
				\grf_{25}&1&\cdot&1&1&\cdot&\cdot&\cdot&\cdot&\cdot&\cdot\\
				\grf_{26}&1&1&1&1&\cdot&\cdot&\cdot&\cdot&\cdot&\cdot\\
				\grf_{29}&\cdot&\cdot&\cdot&\cdot&1&1&\cdot&\cdot&\cdot&\cdot\\
				\grf_{31}&\cdot&\cdot&\cdot&\cdot&1&1&\cdot&\cdot&\cdot&\cdot\\
				\grf_{34}&1&1&1&1&\cdot&\cdot&\cdot&\cdot&\cdot&\cdot\\
				\grf_{35}&2&1&1&\cdot&\cdot&\cdot&\cdot&\cdot&\cdot&\cdot\\
				\grf_{40}&2&1&1&1&\cdot&\cdot&\cdot&\cdot&\cdot&\cdot\\
				\grf_{42}&\cdot&\cdot&\cdot&\cdot&\cdot&\cdot&1&\cdot&1&\cdot\\
				\grf_{43}&2&1&2&1&\cdot&\cdot&\cdot&\cdot&\cdot&\cdot\\
				\grf_{44}&\cdot&\cdot&\cdot&\cdot&\cdot&\cdot&\cdot&\cdot&\cdot&1\\
				\end{block}
				\end{blockarray}
				$\end{center}
			\vspace*{-0.2cm}
			The upper part of the matrix is labeled by $\grf_1$, $\grf_2$, $\grf_3$, $\grf_5$, $\grf_9$, $\grf_{10}$, $\grf_{17}$,  
			$\grf_{18}$,   $\grf_{23}$, and   $\grf_{41}$ together with the characters  $\grf_8$, $\grf_{12}$, $\grf_{19}$,   $\grf_{21}$,  
			$\grf_{27}$,   $\grf_{28}$, $\grf_{30}$, $\grf_{32}$, $\grf_{33}$,  $\grf_{36}$, $\grf_{37}$, $\grf_{38}$, $\grf_{39}$, and  $\grf_{45}$, which are of defect 0.
			\item[\small{C1.5.}]$ab=E(3)cd$ and $(a^2-ab+b^2)(c^2-cd+d^2)\not=0$:
			\vspace*{-0.3cm}
			\begin{center}
			$\fontsize{4}{4}\selectfont\begin{blockarray}[t]{cccc}
					\begin{block}{c[ccc]}
					\grf_5&1&&\\
					\grf_6&&1&\\
					\grf_{13}&&&1\\
					\grf_{35}&&1&1\\
					\grf_{40}&1&1&1\\
					\end{block}
					\end{blockarray}
				$
				\end{center}
			\vspace*{-0.2cm}

			\item[\small{C1.6.}] $e^4=E(3)d^4$, $c^2+ab=0$, and $a^2-ab+b^2\not=0$:
			For $c\not=e$ we have the following form:
			\vspace*{-0.3cm}
			\begin{center}
			$\fontsize{4}{4}\selectfont\begin{blockarray}[t]{ccccccccc}
				\begin{block}{c[cccccccc]}
				&&&&&&&&\\
				\grf_{3}&1&&&&&&&\\
				\grf_{4}&&1&&&&&&\\
				\grf_5&&&1&&&&&\\
				\grf_{6}&&&&1&&&&\\
				\grf_{7}&&&&&1&&&\\
				\grf_{8}&&&&&&1&&\\
				\grf_{10}&&&&&*&*&*&\\
				\grf_{11}&&&&&*&*&&*\\	
				\grf_{16}&1&&&1&&&&\\
				\grf_{35}&1&1&&1&&&&\\
				\grf_{40}&1&1&1&1&&&&\\
				\end{block}
				\end{blockarray}
			$\end{center}
		\vspace*{-0.2cm}
			The $*$ denotes 
			a placeholder for one of two values, 0 or 1. The sum of these values in each line and in each column equals 1.
			For $e=c$ we have a different approach. Notice that $c^3=-d^3$. Since $e=c$ and $e^4=E(3)d^4$ we obtain $c^4=E(3)d^3d$ and, hence, $d=-E(3)^2c$. Due to Proposition \ref{2ir}, the character $\grf_{13}=\grf_{15}$ is not simple. Due to the same proposition, the character $\grf_{14}$ is simple. As a result, the decomposition matrix is of the form:
			\vspace*{-0.6cm}
			\begin{center}
				$\fontsize{4}{4}\selectfont
			\begin{blockarray}[t]{ccccccccccc}
				\begin{block}{c[cccccccccc]}
				\grf_{3}&1&&&&&&&\\
				\grf_{4}&&1&&&&&&\\
				\grf_{6}&&&1&&&&&\\
				\grf_{7}&&&&1&&&&\\
				\grf_{8}&&&&&1&&&\\
				\grf_{10}&&&&*&*&*&&\\
				\grf_{11}&&&&*&*&&*&\\
				\grf_{14}&&&&*&*&*&*&*\\
				\grf_{5}&1&&&&&&&\\
				\grf_{9}&&&&1&&&&\\
				\grf_{13}&1&1&&&&&&\\
				\grf_{15}&1&1&&&&&&\\
				\grf_{16}&1&&1&&&&&\\
				\grf_{18}&1&&1&&&&&\\
				\grf_{33}&1&1&1&&&&&\\
				\grf_{35}&1&1&1&&&&&\\
				\grf_{40}&2&1&1&&&&&\\
				\end{block}
				\end{blockarray}
			$\end{center}
		\vspace*{-0.2cm}
			where $*$ denotes again
			a placeholder for one of two values, 0 or 1. The sum of these values in each line equals 1.
			We also have $\grf_{12}=\grf_{10}$, $\grf_{19}=\grf_{21}$, $\grf_{22}=\grf_{24}$, and $\grf_{28}=\grf_{30}$. The characters $\grf_{10}$ and $\grf_{11}$ are not in the same block.

			\item [\small{C1.7.}]$b=-E(3)a$, $c=-E(3)^2a$, $d=E(3)a$, $e=-a$:
			\vspace*{-0.2cm}
			\begin{center}
			$\fontsize{4}{4}\selectfont
			\begin{blockarray}[t]{cccccccccccccccc}
				\begin{block}{c[ccccccccccccccc]}
				&&&&&&&&\raisebox{-10pt}{{\huge\mbox{{$I_{29}$}}}}&&&&&&\\[-3ex]
				&&&&&&&&&&&&&&\\
				&&&&&&&&&&&&&&\\
				\cmidrule[0.00001cm](lr){2-16}
				\grf_{6}&1&1&\cdot&\cdot&\cdot&\cdot&\cdot&\cdot&\cdot&\cdot&\cdot&\cdot&\cdot&\cdot\\
				\grf_{7}&1&\cdot&1&\cdot&\cdot&\cdot&\cdot&\cdot&\cdot&\cdot&\cdot&\cdot&\cdot&\cdot\\
				\grf_{13} &\cdot&\cdot&1&1&\cdot&\cdot&\cdot&\cdot&\cdot&\cdot&\cdot&\cdot&\cdot&\cdot\\
				\grf_{15}&\cdot&\cdot&\cdot&1&1&\cdot&\cdot&\cdot&\cdot&\cdot&\cdot&\cdot&\cdot&\cdot\\
				\grf_{16}&1&1&1&\cdot&\cdot&\cdot&\cdot&\cdot&\cdot&\cdot&\cdot&\cdot&\cdot&\cdot\\
				\grf_{19}&1&\cdot&1&1&\cdot&\cdot&\cdot&\cdot&\cdot&\cdot&\cdot&\cdot&\cdot&\cdot\\
				\grf_{25}&\cdot&\cdot&1&1&1&\cdot&\cdot&\cdot&\cdot&\cdot&\cdot&\cdot&\cdot&\cdot\\
				\grf_{27}&1&\cdot&1&1&1&\cdot&\cdot&\cdot&\cdot&\cdot&\cdot&\cdot&\cdot&\cdot\\
				\grf_{28}&\cdot&\cdot&\cdot&\cdot&\cdot&\cdot&1&\cdot&1&\cdot&\cdot&\cdot&\cdot&\cdot\\
				\grf_{30}&\cdot&\cdot&\cdot&\cdot&\cdot&1&\cdot&1&\cdot&\cdot&\cdot&\cdot&\cdot&\cdot\\
				\grf_{32}&\cdot&\cdot&\cdot&\cdot&\cdot&1&\cdot&\cdot&\cdot&1&\cdot&\cdot&\cdot&\cdot\\
				\grf_{35}&1&1&1&1&\cdot&\cdot&\cdot&\cdot&\cdot&\cdot&\cdot&\cdot&\cdot&\cdot\\
				\grf_{40}&1&1&1&1&1&\cdot&\cdot&\cdot&\cdot&\cdot&\cdot&\cdot&\cdot&\cdot\\
				\grf_{41}&\cdot&\cdot&\cdot&\cdot&\cdot&\cdot&\cdot&\cdot&\cdot&\cdot&1&1&\cdot&\cdot\\
				\grf_{43}&\cdot&\cdot&\cdot&\cdot&\cdot&1&\cdot&1&\cdot&1&\cdot&\cdot&\cdot&\cdot\\
				\grf_{44}&\cdot&\cdot&\cdot&\cdot&\cdot&\cdot&\cdot&\cdot&\cdot&\cdot&\cdot&\cdot&1&1\\
				\end{block}
				\end{blockarray}
			$\end{center}
		\vspace*{-0.3cm}
			The upper part of the matrix is labeled by $\grf_1$, $\grf_2$, $\grf_3$, $\grf_4$, $\grf_5$, $\grf_8$, $\grf_9$, $\grf_{10}$, $\grf_{11}$,   $\grf_{14}$,  $\grf_{17}$,
			$\grf_{20}$,   $\grf_{21}$, and   $\grf_{22}$ together with  the characters  $\grf_{12}$, $\grf_{18}$, $\grf_{23}$,   $\grf_{24}$,  
			$\grf_{26}$,  $\grf_{29}$, $\grf_{31}$, $\grf_{33}$, $\grf_{34}$, $\grf_{36}$, $\grf_{37}$, $\grf_{38}$, $\grf_{39}$, $\grf_{42}$, and  $\grf_{45}$, which are of defect 0.
			\item[\small{C1.8.}]$b=-E(3)a$, $c=-E(3)^2a$, $d=a$, $e=-E(3)a$:
			\vspace*{-0.2cm}
			\begin{center} 
				$\fontsize{4}{4}\selectfont
				\begin{blockarray}[t]{cccccccccccccc}
				\begin{block}{c[ccccccccccccc]}
				&&&&&&&&&\\
				&&&&&&&\raisebox{-10pt}{{\huge\mbox{{$I_{19}$}}}}&&&&&&\\[-3ex]
				&&&&&&&&&&\\
				&&&&&&&&&&\\
				\cmidrule[0.00001cm](lr){2-14}
				\grf_{4}&1&\cdot&\cdot&\cdot&\cdot&\cdot&\cdot&\cdot&\cdot&\cdot&\cdot&\cdot&\cdot\\
				\grf_{5}&\cdot&1&\cdot&\cdot&\cdot&\cdot&\cdot&\cdot&\cdot&\cdot&\cdot&\cdot&\cdot\\
				\grf_{6}&1&1&\cdot&\cdot&\cdot&\cdot&\cdot&\cdot&\cdot&\cdot&\cdot&\cdot&\cdot\\
				\grf_{7}&1&\cdot&1&\cdot&\cdot&\cdot&\cdot&\cdot&\cdot&\cdot&\cdot&\cdot&\cdot\\
				\grf_{9}&1&1&\cdot&\cdot&\cdot&\cdot&\cdot&\cdot&\cdot&\cdot&\cdot&\cdot&\cdot\\
				\grf_{11}&1&1&\cdot&\cdot&\cdot&\cdot&\cdot&\cdot&\cdot&\cdot&\cdot&\cdot&\cdot\\
				\grf_{13}&1&\cdot&1&\cdot&\cdot&\cdot&\cdot&\cdot&\cdot&\cdot&\cdot&\cdot&\cdot\\
				\grf_{14}&\cdot&\cdot&\cdot&\cdot&1&\cdot&\cdot&\cdot&\cdot&\cdot&\cdot&\cdot&\cdot\\
				\grf_{15}&1&1&\cdot&\cdot&\cdot&\cdot&\cdot&\cdot&\cdot&\cdot&\cdot&\cdot&\cdot\\
				\grf_{16}&1&1&1&\cdot&\cdot&\cdot&\cdot&\cdot&\cdot&\cdot&\cdot&\cdot&\cdot\\
				\grf_{20}&1&1&1&\cdot&\cdot&\cdot&\cdot&\cdot&\cdot&\cdot&\cdot&\cdot&\cdot\\
				\grf_{21}&\cdot&\cdot&\cdot&\cdot&\cdot&\cdot&1&\cdot&\cdot&\cdot&\cdot&\cdot&\cdot\\
				\grf_{22}&1&1&1&\cdot&\cdot&\cdot&\cdot&\cdot&\cdot&\cdot&\cdot&\cdot&\cdot\\
				\grf_{24}&\cdot&\cdot&\cdot&\cdot&\cdot&\cdot&\cdot&1&\cdot&\cdot&\cdot&\cdot&\cdot\\
				\grf_{25}&1&1&1&\cdot&\cdot&\cdot&\cdot&\cdot&\cdot&\cdot&\cdot&\cdot&\cdot\\
				\grf_{27}&2&1&1&\cdot&\cdot&\cdot&\cdot&\cdot&\cdot&\cdot&\cdot&\cdot&\cdot\\
				\grf_{29}&\cdot&\cdot&\cdot&\cdot&1&1&\cdot&\cdot&\cdot&\cdot&\cdot&\cdot&\cdot\\
				\grf_{31}&\cdot&\cdot&\cdot&\cdot&\cdot&\cdot&\cdot&\cdot&\cdot&\cdot&1&\cdot&\cdot\\
				\grf_{32}&\cdot&\cdot&\cdot&\cdot&\cdot&\cdot&\cdot&\cdot&\cdot&\cdot&\cdot&1&\cdot\\
				\grf_{33}&\cdot&\cdot&\cdot&1&\cdot&1&\cdot&\cdot&\cdot&\cdot&\cdot&\cdot&\cdot\\
				\grf_{34}&\cdot&\cdot&\cdot&\cdot&\cdot&\cdot&\cdot&\cdot&\cdot&1&\cdot&\cdot&\cdot\\
				\grf_{35}&2&1&1&\cdot&\cdot&\cdot&\cdot&\cdot&\cdot&\cdot&\cdot&\cdot&\cdot\\
				\grf_{40}&2&2&1&\cdot&\cdot&\cdot&\cdot&\cdot&\cdot&\cdot&\cdot&\cdot&\cdot\\
				\grf_{41}&\cdot&\cdot&\cdot&\cdot&\cdot&\cdot&\cdot&1&1&\cdot&\cdot&\cdot&\cdot\\
				\grf_{43}&\cdot&\cdot&\cdot&\cdot&1&\cdot&\cdot&\cdot&\cdot&\cdot&\cdot&1&\cdot\\
				\grf_{44}&\cdot&\cdot&\cdot&\cdot&\cdot&\cdot&\cdot&1&1&\cdot&\cdot&\cdot&\cdot\\
				\grf_{45}&\cdot&\cdot&\cdot&\cdot&\cdot&\cdot&\cdot&\cdot&\cdot&\cdot&\cdot&\cdot&1\\
				\end{block}
				\end{blockarray}
				$\end{center}
			\vspace*{-0.2cm}
			The upper part of the matrix is labeled by  $\grf_1$, $\grf_2$, $\grf_3$, $\grf_8$,  $\grf_{10}$, $\grf_{12}$, $\grf_{17}$,  
			$\grf_{18}$,  $\grf_{19}$, $\grf_{26}$,  $\grf_{29}$,  $\grf_{30}$, and   $\grf_{42}$ together with the defect 0 characters  $\grf_{23}$, $\grf_{28}$,  $\grf_{36}$, $\grf_{37}$, $\grf_{38}$, and $\grf_{39}$.
	\end{itemize}
	\item[\small{C2.}]$r^2+\gra\grb=0$, for some $\gra,\grb\in\{a,b,c,d,e\}$ distinct: Let $r^2=-ab$. We type:
	\footnotesize{\begin{verbatim}
		gap> t:=List(T,i->List(i,j->Value(j,["w", E(8)*Mvp("x")^(3/2)*Mvp("y")^(3/2)*Mvp("z")^-1*Mvp("t")^-1])));;
		gap> t[40]=t[6]+t[25];
		true
		\end{verbatim}}\normalsize
	\noindent
	We follow now Sections \ref{ddd2} and \ref{dd3} and, up to permutation of the characters, we have:
	\begin{itemize}[leftmargin=0.9cm]
		\item[\small{C2.1.}]$(a^2-ab+b^2)\not=0$ and $(c^2+de)(d^2+ce)(e^2+dc)=0$: This case coincide with C1.3-C1.7.

			 \item[\small{C2.2.}]$(a^2-ab+b^2)\not=0$ and $(c^2+de)(d^2+ce)(e^2+dc)=0$:
			 Without loss of generality, we assume that $c^2+de=0$. Following Section \ref{dd3}, we have that $\grf_{25}=\grf_3+\grf_{15}$. We distinguish the following cases, depending on the behavior of the character $\grf_{15}$.
			 \begin{itemize}[leftmargin=-0.08cm]
			 	\item[] $\mathbf{d^2-de+e^2\not=0}$: Due to Proposition \ref{2ir} the character $\grf_{15}$ is simple.
			 	The characters $\grf_{6}$ and $\grf_{15}$ correspond to the same module, if and only if the matrix models are the same. One can find these matrix models in GAP and notice that they are the same if and only if $ab=de$ and $a+b=d+e$. We multiply the second equation with $a$ and we have $a^2+ab=ad+ae\Rightarrow a^2+de=ad+ae\Rightarrow a\in\{d,e\}$. Since $a+b=d+e$ we obtain $\{a=d,b=e\}$ or $\{a=e,b=d\}$. Therefore, it will be sufficient to consider the cases
			  $\{a,b\}\not=\{d,e\}$ and $a,b\in\{d,e\}$. The decomposition matrices have respectively the following forms:
			   \vspace*{-0.5cm}
			 	\begin{center}
			 	$\tiny{
			 	\begin{blockarray}[t]{cccc}
			 		\begin{block}{c[ccc]}
			 		\grf_{3}&1&&\\
			 		\grf_{6}&&1&\\
			 		\grf_{15}&&&1\\
			 		\grf_{25}&1&&1\\
			 		\grf_{40}&1&1&1\\
			 		\end{block}
			 		\end{blockarray}}
			 	$\;\;\;\;\;\;\;\;\;
			 	$\tiny{
			 	\begin{blockarray}[t]{cccc}
			 		\begin{block}{c[ccc]}
			 		\grf_{3}&1&\\
			 		\grf_{6}&&1\\
			 		\grf_{15}&&1\\
			 		\grf_{25}&1&1\\
			 		\grf_{40}&1&2\\
			 		\end{block}
			 		\end{blockarray}}
			 	$\end{center}\vspace*{-0.2cm}
			One could examine here the case $a=b=d=e$. Since $c^2+ab=0$ we have $c^2+a^2=0$. The decomposition matrix is the following:\vspace*{-0.3cm}
			 \begin{center}
			 $\fontsize{4}{4}\selectfont
			 \begin{blockarray}[t]{ccccccccccc}
			 	\begin{block}{c[cccccccccc]}
			 	&&&&&&&\\
			 	&&&&\raisebox{-10pt}{{\huge\mbox{{$I_{14}$}}}}&&&\\[-3ex]
			 	&&&&&&&\\
			 	&&&&&&&\\
			 	\cmidrule[0.00001cm](lr){2-9}
			 	\grf_{2}&1&\cdot&\cdot&\cdot&\cdot&\cdot&\cdot&\cdot\\
			 	\grf_{4}&1&\cdot&\cdot&\cdot&\cdot&\cdot&\cdot&\cdot\\
			 	\grf_{5}&1&\cdot&\cdot&\cdot&\cdot&\cdot&\cdot&\cdot\\
			 	\grf_{8}&\cdot&\cdot&1&\cdot&\cdot&\cdot&\cdot&\cdot\\
			 	\grf_{9}&\cdot&\cdot&1&\cdot&\cdot&\cdot&\cdot&\cdot\\
			 	\grf_{10}&\cdot&\cdot&\cdot&1&\cdot&\cdot&\cdot&\cdot\\
			 	\grf_{11}&\cdot&\cdot&1&\cdot&\cdot&\cdot&\cdot&\cdot\\
			 	\grf_{12}&\cdot&\cdot&1&\cdot&\cdot&\cdot&\cdot&\cdot\\
			 	\grf_{13}&\cdot&\cdot&\cdot&1&\cdot&\cdot&\cdot&\cdot\\
			 	\grf_{14}&\cdot&\cdot&\cdot&1&\cdot&\cdot&\cdot&\cdot\\
			 	\grf_{15}&\cdot&\cdot&1&\cdot&\cdot&\cdot&\cdot&\cdot\\
			 	\grf_{16}&\cdot&1&1&\cdot&\cdot&\cdot&\cdot&\cdot\\
			 		\grf_{17}&\cdot&\cdot&\cdot&\cdot&1&\cdot&\cdot&\cdot\\
			 		\grf_{18}&\cdot&\cdot&\cdot&\cdot&1&\cdot&\cdot&\cdot\\
			 	\grf_{19}&\cdot&1&1&\cdot&\cdot&\cdot&\cdot&\cdot\\
			 	\grf_{20}&\cdot&1&1&\cdot&\cdot&\cdot&\cdot&\cdot\\
			 	\grf_{21}&\cdot&\cdot&\cdot&\cdot&1&\cdot&\cdot&\cdot\\
			 	\grf_{22}&\cdot&1&1&\cdot&\cdot&\cdot&\cdot&\cdot\\
			 	\grf_{23}&\cdot&1&1&\cdot&\cdot&\cdot&\cdot&\cdot\\
			 	\grf_{24}&\cdot&\cdot&\cdot&\cdot&1&\cdot&\cdot&\cdot\\
			 	\grf_{25}&\cdot&1&1&\cdot&\cdot&\cdot&\cdot&\cdot\\
			 	\grf_{26}&\cdot&\cdot&\cdot&\cdot&\cdot&1&\cdot&\cdot\\
			 	\grf_{27}&\cdot&\cdot&\cdot&\cdot&\cdot&1&\cdot&\cdot\\
			 	\grf_{29}&\cdot&\cdot&\cdot&\cdot&\cdot&1&\cdot&\cdot\\
			 	\grf_{30}&\cdot&\cdot&\cdot&\cdot&\cdot&1&\cdot&\cdot\\
			 	\grf_{31}&\cdot&\cdot&\cdot&\cdot&\cdot&\cdot&1&\cdot\\
			 	\grf_{32}&\cdot&\cdot&\cdot&\cdot&\cdot&\cdot&1&\cdot\\
			 	\grf_{33}&1&\cdot&\cdot&\cdot&1&\cdot&\cdot&\cdot\\
			 	\grf_{34}&\cdot&\cdot&\cdot&\cdot&\cdot&\cdot&1&\cdot\\
			 	\grf_{35}&\cdot&\cdot&\cdot&\cdot&\cdot&\cdot&1&\cdot\\
			 	\grf_{40}&\cdot&1&2&\cdot&\cdot&\cdot&1&\cdot\\
			 		\grf_{42}&\cdot&\cdot&\cdot&\cdot&\cdot&\cdot&\cdot&1\\
			 		\grf_{44}&\cdot&\cdot&\cdot&\cdot&\cdot&\cdot&\cdot&1\\
			 		\grf_{45}&\cdot&\cdot&\cdot&\cdot&\cdot&\cdot&\cdot&1\\
			 	\end{block}
			 	\end{blockarray}
			 $\end{center}\vspace*{-0.2cm}
			 The upper part of the matrix is indexed by  $\grf_1$, $\grf_3$, $\grf_6$, $\grf_7$, $\grf_{17}$, $\grf_{26}$, $\grf_{31}$, $\grf_{41}$, together with the characters $\grf_{28}$, $\grf_{36}$, $\grf_{37}$, $\grf_{38}$, $\grf_{39}$, $\grf_{43}$, which are of defect 0.
			 \item[] $\mathbf{d^2-de+e^2=0}$:\vspace*{-0.5cm}
			 \begin{center}
			 	$\fontsize{4}{4}\selectfont
			 	\begin{blockarray}[t]{ccccccc}
			 	\begin{block}{c[cccccc]}
			 	\grf_{3}&1&&&\\
			 	\grf_{4}&&1&&\\
			 	\grf_{5}&&&1&\\
			 	\grf_{6}&&&&1\\
			 	\grf_{15}&&1&1&\\
			 	\grf_{25}&1&1&1&\\
			 	\grf_{40}&1&1&1&1\\
			 	\end{block}
			 	\end{blockarray}
			 	$\end{center}
			 \vspace*{-0.2cm}
			 \end{itemize}
			 \item[\small{C2.3.}]$(a^2-ab+b^2)=0$ and $(c^2+de)(d^2+ce)(e^2+dc)\not=0$:
			 \vspace*{-0.2cm}\begin{center}
			 	$\tiny{
			 	\begin{blockarray}[t]{cccc}
			 		\begin{block}{c[ccc]}
			 		\grf_{1}&1&&\\
			 		\grf_{2}&&1&\\
			 		\grf_{6}&1&1&\\
			 		\grf_{25}&&&1\\
			 		\grf_{40}&1&1&1\\
			 		\end{block}
			 		\end{blockarray}}
			 	$
			 \end{center}\vspace*{-0.3cm}
			 
			 \item[\small{C2.4.}]$(a^2-ab+b^2)(c^2+de)(d^2+ce)(e^2+dc)\not=0$:
			 \vspace*{-0.2cm}\begin{center}
			 	$\tiny{
			 	\begin{blockarray}[t]{ccc}
			 		\begin{block}{c[cc]}
			 		\grf_{6}&1&\\
			 		\grf_{25}&&1\\
			 		\grf_{40}&1&1\\
			 		\end{block}
			 		\end{blockarray}}
			 	$
			 	\vspace*{-0.4cm}
			 \end{center}
			 \end{itemize}
		\end{itemize}
\subsubsection{The 6-dimensional characters}
\label{sixi}
Each 6-dimensional character depends on a value $\gra\in\{a,b,c,d,e\}$. Without loss of generality, we examine $\grf_{43}$, which depends on $c$.  We focus on the values of $a,b,c,d,e$ that annihilate the associated Schur element $s_{\grf_{43}}$.
We encounter four types of factors in the factorization of  $s_{\grf_{43}}$. 
The first factor is $(c-a)(c-b)(c-d)(c-e)$, which is annihilated when $c=\grb$, for $\grb\in\{a,b,d,e\}$. If this happens, the character $\grf_{43}$ coincides with the 6-dimensional character depending on $\grb$. We examine now the rest of the factors.
\begin{itemize}[leftmargin=0.7cm]
	\item[\small{C1.}] $(ab+de)(ad+be)(ae+bd)=0$: We assume $ab=-de$. The characters $\grf_{16}$ and $\grf_{25}$  correspond to distinct modules, since $ab=-de$ and, hence, the associated matrix models are different. 
	 \footnotesize{\begin{verbatim}
	 	gap> t:=List(T,i->List(i,j->Value(j,["x", E(20)*Mvp("y")^-1*Mvp("t")*Mvp("w")])));;
	 	gap> t[43]=t[16]+t[25];
	 	true
	 	\end{verbatim}}\normalsize
	 \noindent
	 We follow now Section \ref{dd3} and we distinguish the following cases:
	 \begin{itemize}[leftmargin=0.95cm]
	\item[\small{C1.1.}]
	 	Neither $\grf_{16}$ nor $\grf_{25}$ is simple. We need to examine the following two cases:
	 	\begin{itemize}[leftmargin=-0.007cm]
	 	\item[]$\mathbf{a^2+bc=d^2+ce=0}$\textbf{:} Since $d^2=-ce$ and $a^2=-bc$, we obtain $a^2e=bd^2$. If $a=d$, then $e=b$ and, therefore, $ab=de$, which contradicts the fact that $ab=-de$. As a result, $a\not=d$ and, hence, $\grf_1\not=\grf_4$.
	 	Let $(b^2-bc+c^2)(c^2-ce+e^2)\not=0$. We provide the decomposition matrix for the cases $b=e$ and $b\not=e$, respectively. Notice that in the latter case, 
	 the characters $\grf_{10}$ and $\grf_{14}$ correspond to distinct modules.\vspace*{-0.2cm}
	 \begin{center}
	 	$\fontsize{4}{4}\selectfont\begin{blockarray}[t]{ccccc}
	 			\begin{block}{c[cccc]}
	 			\grf_{1}&1&&&\\
	 			\grf_{4}&&1&&\\
	 			\grf_{10}&&&1&\\
	 			\grf_{14}&&&&1\\
	 			\grf_{16}&1&&1&\\
	 			\grf_{25}&&1&&1\\
	 			\grf_{43}&1&1&1&1\\
	 			\end{block}
	 			\end{blockarray}
	 		$\quad \quad \quad \quad
	 		$\fontsize{4}{4}\selectfont
	 		\begin{blockarray}[t]{cccc}
	 				\begin{block}{c[ccc]}
	 					\grf_{1}&1&&\\
	 					\grf_{4}&&1&\\
	 					\grf_{10}&&&1\\
	 					\grf_{14}&&&1\\
	 					\grf_{16}&1&&1\\
	 					\grf_{25}&&1&1\\
	 					\grf_{43}&1&1&2\\
	 				\end{block}
	 			\end{blockarray}$
	 		\end{center}
	 		\vspace*{-0.2cm}
	 	We assume now that $b^2-bc+c^2=0$ and let $b=-E(3)c$. Since $a^2=-bc$ we have  $a=\pm E(3)^2c$. As a result, the characters $\grf_1$, $\grf_2$ and $\grf_3$ are distinct. We now focus on the behavior of character $\grf_{14}$ and we first assume that 
	 $c^2-ce+e^2=0$. Therefore, we have $b=-E(3)c$, $a=E(3)^2c$, $e=-E(3)c$ and $d=-E(3)^2c$.
	 	\vspace*{-0.3cm}
	 	\begin{center}
	 		$\fontsize{4}{4}\selectfont
	 		\begin{blockarray}[t]{cccccccccccccc}
	 		\begin{block}{c[ccccccccccccc]}
	 		&&&&&&&&&&\\
	 		&&&&&&&&&&\\
	 		&&&&&\raisebox{-10pt}{{\huge\mbox{{$I_{24}$}}}}&&&&&\\[-3ex]
	 		&&&&&&&&&&\\
	 		&&&&&&&&&\\
	 		&&&&&&&&&\\
	 		\cmidrule[0.00001cm](lr){2-11}
	 		\grf_{5}&\cdot&1&\cdot&\cdot&\cdot&\cdot&\cdot&\cdot&\cdot&\cdot\\
	 		\grf_{6}&1&1&\cdot&\cdot&\cdot&\cdot&\cdot&\cdot&\cdot&\cdot\\
	 		\grf_{9}&1&1&\cdot&\cdot&\cdot&\cdot&\cdot&\cdot&\cdot&\cdot\\
	 		\grf_{10}&\cdot&1&1&\cdot&\cdot&\cdot&\cdot&\cdot&\cdot&\cdot\\
	 		\grf_{13}&\cdot&\cdot&1&1&\cdot&\cdot&\cdot&\cdot&\cdot&\cdot\\
	 		\grf_{14}&\cdot&1&1&\cdot&\cdot&\cdot&\cdot&\cdot&\cdot&\cdot\\
	 		\grf_{15}&\cdot&\cdot&\cdot&\cdot&\cdot&1&\cdot&\cdot&\cdot&\cdot\\
	 		\grf_{16}&1&1&1&\cdot&\cdot&\cdot&\cdot&\cdot&\cdot&\cdot\\
	 		\grf_{20}&1&1&1&\cdot&\cdot&\cdot&\cdot&\cdot&\cdot&\cdot\\
	 		\grf_{21}&\cdot&\cdot&\cdot&\cdot&\cdot&\cdot&1&\cdot&\cdot&\cdot\\
	 		\grf_{22}&\cdot&1&1&1&\cdot&\cdot&\cdot&\cdot&\cdot&\cdot\\
	 		\grf_{25}&\cdot&1&1&1&\cdot&\cdot&\cdot&\cdot&\cdot&\cdot\\
	 		\grf_{27}&1&1&1&1&\cdot&\cdot&\cdot&\cdot&\cdot&\cdot\\
	 		\grf_{30}&1&1&1&1&\cdot&\cdot&\cdot&\cdot&\cdot&\cdot\\
	 		\grf_{32}&\cdot&\cdot&\cdot&\cdot&1&1&\cdot&\cdot&\cdot&\cdot\\
	 		\grf_{34}&1&2&1&\cdot&\cdot&\cdot&\cdot&\cdot&\cdot&\cdot\\
	 		\grf_{35}&\cdot&\cdot&\cdot&\cdot&1&1&\cdot&\cdot&\cdot&\cdot\\
	 		\grf_{40}&1&2&1&1&\cdot&\cdot&\cdot&\cdot&\cdot&\cdot\\
	 		\grf_{41}&\cdot&\cdot&\cdot&\cdot&\cdot&\cdot&\cdot&1&1&\cdot\\
	 		\grf_{43}&1&2&2&1&\cdot&\cdot&\cdot&\cdot&\cdot&\cdot\\
	 		\grf_{45}&\cdot&\cdot&\cdot&\cdot&\cdot&\cdot&\cdot&\cdot&\cdot&1\\
	 		\end{block}
	 		\end{blockarray}
	 		$\end{center}\vspace*{-0.2cm}
	 	The upper part of the matrix is labeled by  $\grf_1$, $\grf_2$, $\grf_3$, $\grf_4$,  $\grf_{7}$, $\grf_{11}$, $\grf_{17}$,  
	 	$\grf_{18}$,  $\grf_{19}$, $\grf_{26}$,  $\grf_{29}$,   and   $\grf_{42}$ together with the characters $\grf_8$, $\grf_{12}$, $\grf_{23}$, $\grf_{24}$, $\grf_{26}$, $\grf_{28}$, $\grf_{29}$, $\grf_{31}$, $\grf_{33}$,  $\grf_{36}$, $\grf_{37}$, $\grf_{38}$, $\grf_{39}$, and $\grf_{44}$, which are of defect 0.
	 	
	 	It remains to examine the case $c^2-ce+e^2\not=0$. We recall that $a,b$ and $c$ are distinct, $d\not=a$, and that $b=-E(3)c$. If $d=b$, then we use the equation $d^2+ce=0$ and we obtain $e=-E(3)^2c$,  which contradicts the fact that $c^2-ce+e^2\not=0$. As a result, $d\not=b$. If $d\not=c$, then
	 	 the decomposition matrix is of the following form:\vspace*{-0.2cm}
	 	 \begin{center}
	 	$\fontsize{4}{4}\selectfont
	 	\begin{blockarray}[t]{cccccccc}
	 		\begin{block}{c[ccccccc]}
	 		\grf_{1}&1&&&&\\
	 		\grf_{2}&&1&&&\\
	 		\grf_{3}&&&1&&\\
	 		\grf_{4}&&&&1&\\
	 		\grf_{14}&&&&&1\\
	 		\grf_{10}&&1&1&&\\
	 		\grf_{25}&&&&1&1\\
	 		\grf_{43}&1&1&1&1&1\\
	 		\end{block}
	 		\end{blockarray}
	 	$\end{center}\vspace*{-0.3cm}
	 	We assume now that $b=c$. Therefore, we have the case $a=-E(3)^2c, b=-E(3)c, d=c$, and $e=-c$:
	 	\vspace*{-0.6cm}
	 	\begin{center}
	 		$\fontsize{4}{4}\selectfont
	 		\begin{blockarray}[t]{cccccccccccccc}
	 		\begin{block}{c[ccccccccccccc]}
	 		&&&&&&&&&&\\
	 		&&&&&&&&&&\\
	 		&&&&&&\raisebox{-10pt}{{\huge\mbox{{$I_{24}$}}}}&&&&\\[-3ex]
	 		&&&&&&&&&&\\
	 		&&&&&&&&&\\
	 		&&&&&&&&&\\
	 		\cmidrule[0.00001cm](lr){2-12}
	 		\grf_{4}&\cdot&\cdot&1&\cdot&\cdot&\cdot&\cdot&\cdot&\cdot&\cdot&\cdot\\
	 		\grf_{7}&1&\cdot&1&\cdot&\cdot&\cdot&\cdot&\cdot&\cdot&\cdot&\cdot\\
	 		\grf_{8}&1&\cdot&1&\cdot&\cdot&\cdot&\cdot&\cdot&\cdot&\cdot&\cdot\\
	 		\grf_{10}&\cdot&1&1&\cdot&\cdot&\cdot&\cdot&\cdot&\cdot&\cdot&\cdot\\
	 		\grf_{11}&\cdot&1&1&\cdot&\cdot&\cdot&\cdot&\cdot&\cdot&\cdot&\cdot\\
	 		\grf_{15}&\cdot&\cdot&\cdot&\cdot&\cdot&\cdot&1&\cdot&\cdot&\cdot&\cdot\\
	 		\grf_{16}&1&1&1&\cdot&\cdot&\cdot&\cdot&\cdot&\cdot&\cdot&\cdot\\
	 		\grf_{17}&1&1&1&\cdot&\cdot&\cdot&\cdot&\cdot&\cdot&\cdot&\cdot\\
	 		\grf_{21}&\cdot&\cdot&\cdot&\cdot&\cdot&\cdot&\cdot&1&\cdot&\cdot&\cdot\\
	 		\grf_{24}&\cdot&\cdot&\cdot&\cdot&\cdot&\cdot&\cdot&\cdot&1&\cdot&\cdot\\
	 		\grf_{25}&\cdot&\cdot&1&\cdot&\cdot&\cdot&1&\cdot&\cdot&\cdot&\cdot\\
	 			\grf_{26}&\cdot&\cdot&\cdot&\cdot&1&1&\cdot&\cdot&\cdot&\cdot&\cdot\\
	 			\grf_{27}&1&\cdot&1&\cdot&\cdot&\cdot&1&\cdot&\cdot&\cdot&\cdot\\
	 		\grf_{29}&\cdot&\cdot&\cdot&\cdot&\cdot&\cdot&\cdot&\cdot&\cdot&1&\cdot\\
	 		\grf_{31}&\cdot&1&1&\cdot&\cdot&\cdot&1&\cdot&\cdot&\cdot&\cdot\\
	 		\grf_{32}&\cdot&\cdot&\cdot&1&\cdot&1&\cdot&\cdot&\cdot&\cdot&\cdot\\
	 			\grf_{34}&\cdot&\cdot&\cdot&\cdot&\cdot&\cdot&\cdot&\cdot&\cdot&\cdot&1\\
	 		\grf_{35}&1&1&2&\cdot&\cdot&\cdot&\cdot&\cdot&\cdot&\cdot&\cdot\\
	 		\grf_{40}&1&\cdot&1&\cdot&\cdot&\cdot&1&\cdot&\cdot&\cdot&\cdot\\
	 			\grf_{43}&1&1&2&\cdot&\cdot&\cdot&1&\cdot&\cdot&\cdot&\cdot\\
	 			\grf_{44}&1&1&2&\cdot&\cdot&\cdot&1&\cdot&\cdot&\cdot&\cdot\\
	 			\grf_{45}&\cdot&\cdot&\cdot&1&1&1&\cdot&\cdot&\cdot&\cdot&\cdot\\
	 		\end{block}
	 		\end{blockarray}
	 		$\end{center}\vspace*{-0.2cm}
	 	The upper part of the matrix is labeled by  $\grf_1$, $\grf_2$, $\grf_3$, $\grf_9$,  $\grf_{12}$, $\grf_{13}$, $\grf_{14}$,  
	 	$\grf_{20}$,  $\grf_{23}$, $\grf_{28}$,   and   $\grf_{33}$ together with the characters $\grf_5$, $\grf_{6}$, $\grf_{18}$, $\grf_{19}$, $\grf_{22}$, $\grf_{30}$, $\grf_{36}$, $\grf_{37}$, $\grf_{38}$,  $\grf_{39}$, $\grf_{41}$, and $\grf_{42}$, which are of defect 0.

	 		\item[]$\mathbf{c^2+ab=d^2+ce=0}$\textbf{:}
	 		We recall that $ab=-de$ and we obtain $c^2=de$. Moreover, from the fact that $d^2=-ce$ it follows $d^3=-c^3$. Since $d\not=c$, the characters $\grf_{3}$ and $\grf_4$ correspond to distinct modules. Let $c=-E(3)^{\grk}d$, with $\grk\in\{0,1,2\}$. Since $d^2=-ce$ we obtain $e=E(3)^{-\grk}d$ and $ab=-E(3)^{-\grk}d^2$.
	 		We first examine the case $\grk\not=0$. Without loss of generality, let $\grk=1$. Hence, $c^2-ce+e^2=0$ and, due to Proposition \ref{2ir} the character $\grf_{14}$ is simple. 	If $a^2-ab+b^2=0$, then the decomposition matrix is of the form described in C1.3 and C1.7 of section \ref{dimm5}.
	 			If $a^2-ab+b^2\not=0$, then the decomposition matrix is of the following form:\vspace*{-0.4cm}
	 			\begin{center}
	 			$\fontsize{4}{4}\selectfont
	 			\begin{blockarray}[t]{ccccc}
	 				\begin{block}{c[cccc]}
	 				\grf_{3}&1&&&\\
	 				\grf_{4}&&1&&\\
	 				\grf_{5}&&&1&\\
	 				\grf_{6}&&&&1\\
	 				\grf_{14}&1&&1&\\
	 				\grf_{16}&1&&&1\\
	 				\grf_{25}&1&1&1&\\
	 				\grf_{43}&2&1&1&1\\
	 				\end{block}
	 				\end{blockarray}
	 			$
	 		\end{center}\vspace*{-0.2cm}
	 		 We now assume that $\grk=0$. As a result, $(a^2-ab+b^2)(c^2-ce+e^2)\not=0$.  If $a\not\in\{c,e\}$ or $b\not\in\{c,e\}$ the characters $\grf_{6}$ and $\grf_{14}$ correspond to distinct modules. Therefore, we have:\vspace*{-0.6cm}
	 		 \begin{center}
	 			$\fontsize{4}{4}\selectfont\begin{blockarray}[t]{cccccccc}
	 				\begin{block}{c[ccccccc]}
	 				\grf_{3}&1&&&&&\\
	 				\grf_{4}&&1&&&&\\
	 				\grf_{6}&&&1&&&\\
	 				\grf_{14}&&&&1&&\\
	 				\grf_{16}&1&&1&&&\\
	 				\grf_{25}&&1&&1&&\\
	 				\grf_{43}&1&1&1&1&&\\
	 				&&&&&&\\
	 				\end{block}
	 				\end{blockarray}
	 			$\end{center}\vspace*{-0.2cm}
	 	We assume now $a,b\in\{c,e\}$. Suppose that $a=b=c$. Since $c=-d$ and $e=d$ we have $ab=c^2=de$, which contradicts the fact that $ab=-de$. Similarly, for $a=b=e$. As a result, $a\not=b$ and, hence, we can restrict ourselves to examining the case $a=c=-d$ and $b=e=d$. \vspace*{-0.5cm}
	 	\begin{center}
	 			 $
	 			 \fontsize{4}{4}\selectfont
	 			 \begin{blockarray}[t]{ccccccccccc}
	 			 	\begin{block}{c[cccccccccc]}
	 			 	&&&&&&&\\
	 			 	&&&&\raisebox{-10pt}{{\huge\mbox{{$I_{14}$}}}}&&&\\[-3ex]
	 			 	&&&&&&&\\
	 			 	&&&&&&&\\
	 			 	\cmidrule[0.00001cm](lr){2-9}
	 			 	\grf_{3}&1&\cdot&\cdot&\cdot&\cdot&\cdot&\cdot&\cdot\\
	 			 	\grf_{4}&\cdot&1&\cdot&\cdot&\cdot&\cdot&\cdot&\cdot\\
	 			 	\grf_{5}&\cdot&1&\cdot&\cdot&\cdot&\cdot&\cdot&\cdot\\
	 			 	\grf_{8}&\cdot&\cdot&1&\cdot&\cdot&\cdot&\cdot&\cdot\\
	 			 	\grf_{9}&\cdot&\cdot&1&\cdot&\cdot&\cdot&\cdot&\cdot\\
	 			 	\grf_{10}&\cdot&\cdot&1&\cdot&\cdot&\cdot&\cdot&\cdot\\
	 			 	\grf_{12}&\cdot&\cdot&\cdot&1&\cdot&\cdot&\cdot&\cdot\\
	 			 	\grf_{13}&\cdot&\cdot&1&\cdot&\cdot&\cdot&\cdot&\cdot\\
	 			 	\grf_{14}&\cdot&\cdot&1&\cdot&\cdot&\cdot&\cdot&\cdot\\
	 			 	\grf_{15}&\cdot&\cdot&\cdot&1&\cdot&\cdot&\cdot&\cdot\\
	 			 	\grf_{16}&1&\cdot&1&\cdot&\cdot&\cdot&\cdot&\cdot\\
	 			 	\grf_{17}&\cdot&1&1&\cdot&\cdot&\cdot&\cdot&\cdot\\
	 			 		\grf_{18}&\cdot&1&1&\cdot&\cdot&\cdot&\cdot&\cdot\\
	 			 			\grf_{19}&1&\cdot&1&\cdot&\cdot&\cdot&\cdot&\cdot\\
	 			 		\grf_{20}&1&\cdot&1&\cdot&\cdot&\cdot&\cdot&\cdot\\
	 			 			\grf_{21}&\cdot&1&1&\cdot&\cdot&\cdot&\cdot&\cdot\\
	 			 			\grf_{22}&\cdot&1&1&\cdot&\cdot&\cdot&\cdot&\cdot\\
	 			 			\grf_{23}&\cdot&1&1&\cdot&\cdot&\cdot&\cdot&\cdot\\
	 			 			\grf_{25}&\cdot&1&1&\cdot&\cdot&\cdot&\cdot&\cdot\\
	 			 	\grf_{27}&1&1&1&\cdot&\cdot&\cdot&\cdot&\cdot\\
	 			 	\grf_{28}&\cdot&\cdot&\cdot&\cdot&1&\cdot&\cdot&\cdot\\
	 			 		\grf_{29}&1&1&1&\cdot&\cdot&\cdot&\cdot&\cdot\\
	 			 			\grf_{30}&1&1&1&\cdot&\cdot&\cdot&\cdot&\cdot\\
	 			 	\grf_{33}&\cdot&\cdot&\cdot&\cdot&\cdot&1&\cdot&\cdot\\
	 			 	\grf_{34}&\cdot&\cdot&\cdot&\cdot&\cdot&\cdot&1&\cdot\\
	 			 	\grf_{35}&\cdot&\cdot&\cdot&\cdot&\cdot&\cdot&1&\cdot\\
	 			 	\grf_{41}&1&1&2&\cdot&\cdot&\cdot&\cdot&\cdot\\
	 			 	\grf_{43}&1&1&2&\cdot&\cdot&\cdot&\cdot&\cdot\\
	 			 	\grf_{44}&\cdot&\cdot&\cdot&\cdot&\cdot&\cdot&\cdot&1\\
	 			 	\grf_{45}&\cdot&\cdot&\cdot&\cdot&\cdot&\cdot&\cdot&1\\
	 			 	\end{block}
	 			 	\end{blockarray}
	 			 $\end{center}\vspace*{-0.2cm}
	 			 where the upper part of the matrix is indexed by  $\grf_1$, $\grf_2$, $\grf_6$, $\grf_{11}$, $\grf_{26}$, $\grf_{31}$, $\grf_{32}$, and $\grf_{42}$, together with the characters $\grf_7$, $\grf_{24}$, $\grf_{36}$, $\grf_{37}$, $\grf_{38}$, $\grf_{39}$, which are of defect 0.
	 			\end{itemize}
	 			\item[\small{C1.2.}]
	 			Only one of the characters $\grf_{16}$ and $\grf_{25}$ is simple. Without loss of generality and following Section \ref{dd3} we suppose that $c^2=-ab$. It remains to examine two cases; $a^2-ab+b^2\not=0$ and $a^2-ab+b^2=0$. The decomposition matrices are respectively the following:\vspace*{-0.6cm}
	 			\begin{center}
	 				$\tiny{
	 				\begin{blockarray}[t]{cccc}
	 				\begin{block}{c[ccc]}
	 				\grf_{3}&1&&\\
	 				\grf_{6}&&1&\\
	 				\grf_{25}&&&1\\
	 				\grf_{16}&1&1&\\
	 				\grf_{43}&1&1&1\\
	 				\end{block}
	 				\end{blockarray}}
	 				$\quad \quad\quad
	 				$\tiny{
	 				\begin{blockarray}[t]{ccccc}
	 				\begin{block}{c[cccc]}
	 				\grf_{1}&1&&&\\
	 				\grf_{2}&&1&&\\
	 				\grf_{3}&&&1&\\
	 				\grf_{25}&&&&1\\
	 				\grf_{6}&1&1&&\\
	 				\grf_{16}&1&1&1&\\
	 				\grf_{43}&1&1&1&1\\
	 				\end{block}
	 				\end{blockarray}}
	 				$\end{center}\vspace*{-0.3cm}
	 			\item[\small{C1.3.}]The characters $\grf_{16}$ and $\grf_{25}$ are simple. We recall that $\grf_{43}=\grf_{16}+\grf_{25}$. Therefore, the decomposition matrix is of the following form:
	 			\vspace*{-0.2cm}
	 			\begin{center}
	 				$\tiny{
	 				\begin{blockarray}[t]{ccc}
	 				\begin{block}{c[cc]}
	 				\grf_{16}&1&\\
	 				\grf_{25}&&1\\
	 				\grf_{43}&1&1\\
	 				\end{block}
	 				\end{blockarray}}
	 				$
	 			\end{center}\vspace*{-0.3cm}
	 			
	 	\item[\small{C2}.] $(a^2c-bde)(b^2c-ade)(d^2c-abe)(e^2c-abd)=0$:	Without loss of generality, we suppose that $e^2c=abd$. We type:
	 	\footnotesize{\begin{verbatim}
	 		gap> t:=List(T,i->List(i,j->Value(j,["t", -Mvp("w")^2*Mvp("y")^-1*Mvp("z")*Mvp("x")^-1])));;
	 		gap> t[43]=t[14]+t[35];
	 		true
	 		\end{verbatim}}\normalsize
	 	\noindent
	 	We distinguish the following cases, based on Sections \ref{ddd2} and \ref{sout}. 
	 	If the characters $\grf_{35}$ and $\grf_{14}$ are not simple, we encounter Cases C1.2 and C1.3. On the other hand, if
	 	the characters $\grf_{14}$ and $\grf_{35}$ are simple we have the following form: \vspace*{-0.2cm}\begin{center}
	 		$\tiny{	 	\begin{blockarray}[t]{cccc}
	 			\begin{block}{c[ccc]}
	 			\grf_{14}&1&&\\
	 			\grf_{30}&&1&\\
	 				\grf_{35}&&&1\\
	 			\grf_{43}&1&&1\\
	 			\end{block}
	 			\end{blockarray}}
	 		$\end{center}\vspace*{-0.2cm}
	 		We consider now the case where
	 		 $\grf_{35}$ is simple, while  $\grf_{14}$ is not. The decomposition matrix is of the form:
	 		\vspace*{-0.6cm}\begin{center}
	 			$\tiny{
	 			\begin{blockarray}[t]{cccc}
	 			\begin{block}{c[ccc]}
	 			\grf_{3}&1&&\\
	 			\grf_{5}&&1&\\
	 			\grf_{35}&&&1\\
	 			\grf_{14}&1&1&\\
	 			\grf_{43}&1&1&1\\
	 			\end{block}
	 			\end{blockarray}}
	 			$\end{center}\vspace*{-0.2cm}
	 	Finally, we consider the case where the character $\grf_{14}$ is simple, while the character $\grf_{35}$ is not. The only new case, apart from the ones studied in Case C1, is 
	 		the case where $\grf_{35}$ breaks up into two 2-dimensional simple characters.
	 		Due to Proposition \ref{2ir} and Case C2 of Section \ref{sout}, we assume that  $ab=E(3)cd$ and $(a^2-ab+b^2)(c^2-cd+d^2)\not=0$. As a result, $\grf_{35}=\grf_6+\grf_{13}$, where $\grf_6$ and $\grf_{13}$ are distinct.  If $ab=ce$, one could use the fact that $e^2c=abd$ in order to obtain $e=d$. However, since $ab=E(3)cd$, we have $ce=E(3)ce$, which is a contradiction. As a result, $ab\not=ce$ and the characters $\grf_6$ and $\grf_{14}$ are distinct. Moreover, $\grf_{13}$ and $\grf_{14}$ are  distinct. These characters correspond to the same module, if $e=d$. Since $e^2c=abd$ and the 
	 		fact that $ab=E(3)cd$, $e=d$ is a contradiction. Summing up, we have:
	 		\vspace*{-0.3cm}\begin{center}
	 		$\tiny{
	 	\begin{blockarray}[t]{cccc}
	 			\begin{block}{c[ccc]}
	 			\grf_{6}&1&&\\
	 			\grf_{13}&&1&\\
	 			\grf_{14}&&&1\\
	 			\grf_{43}&1&1&1\\
	 			\end{block}
	 			\end{blockarray}}
	 		$
	 	\end{center}\vspace*{-0.4cm}
	 \end{itemize}
	 \item[C3.] $c^4+abde=0$: We type:
 \footnotesize{\begin{verbatim}
	 	gap> t:=List(T,i->List(i,j->Value(j,["w", E(4)*Mvp("z")^4*Mvp("y")^-1*Mvp("t")^-1*Mvp("x")^-1])));;
	 	gap> t[43]=t[3]+t[40];
	 	true
	 	\end{verbatim}} \vspace*{-0.4cm}\normalsize
	 \noindent
	 Following section \ref{dimm5} we notice that the only case we have not examined in  C1 and C2 is when $\grf_{40}$ is simple. In this case the decomposition matrix is of the following form: \vspace*{-0.2cm}\begin{center}
	 	$\tiny{
	 	\begin{blockarray}[t]{ccc}
	 		\begin{block}{c[cc]}
	 		\grf_{3}&1&\\
	 		\grf_{40}&&1\\
	 		\grf_{43}&1&1\\
		\end{block}
	 		\end{blockarray}}
	 	$\end{center}\vspace*{-0.7cm}
	 \end{itemize}
	 \begin{rem}
	 	\mbox{}
	 	\vspace*{-\parsep}
	 	\vspace*{-\baselineskip}\\
	 \begin{enumerate}
	 	\item Let $\gru$ be a specialization, such that the algebra $\mathbb{C}H_k$ coincides with the specialized cyclotomic Hecke algebra for $d$-Harish-Chandra series of $G_4$, $G_8$ or $G_{16}$. Using the methodology described in this paper we recover the same decomposition matrices determined in \cite{chlouveraki} by Chlouveraki and Miyachi.
	 	
	 	\item
	 	In \cite{chavli} \S 5 we have classified the simple representations of the braid group $B_3$ for dimension $k=2, 3, 4, 5$. Let $s_1\mapsto A$ and $s_2\mapsto B$ be such a representation. In order to deal with the case $k=5$, we made an assumption for the determinant of the matrix $A$ (which is the same as the determinant of the matrix $B$); we assumed that det$A\not=-\grl_i^6\grl_j^{-1}$, where $\grl_i, \grl_j$ denote any eigenvalues of $A$, not necessarily distinct.
	 	Under this assumption we proved that every 5-dimensional simple $\mathbb{C}B_3$-module coincides with a 5-dimensional simple $\mathbb{C}H_5$-module of defect 0.
	 In Section \ref{sixi} we  consider the case det$A=-\grl_i^6\grl_j^{-1}$ and we obtain again the result that every 5-dimensional simple $\mathbb{C}B_3$-module coincides with a 5-dimensional simple $\mathbb{C}H_5$-module, which is not of defect 0. More precisely, it is in the same block with a 6-dimensional non-simple $\mathbb{C}H_5$-module.
	 	\end{enumerate}
	 \end{rem}

\end{document}